\documentclass[12pt,reqno]{amsart}
\usepackage[left=3cm,top=2cm,right=3cm,bottom=2cm]{geometry}
\usepackage{amsbsy}
\usepackage{tikz}
\usepackage[mathscr]{euscript}
\usepackage{amssymb,amsthm}
\usepackage{multirow}
\usepackage{relsize}
\usepackage{bm}
\usepackage{arydshln}

\AtBeginDocument{%
   \def\MR#1{}
}

\newcommand{\join}{\otimes}

\newcommand{\submat}[1]{ \widetilde { B_{#1} } }
\newcommand{\bigmat}[1]{ \widetilde { B_{#1}^+ } }
\newcommand{\typei}{type-i}
\newcommand{\typeii}{type-ii}
\newcommand{\scwalgo}{\tt diagonalize}
\newcommand{\M}[1]{M_{#1}}

\DeclareMathOperator*{\id}{\mathit{id}}
\DeclareMathOperator*{\cwop}{\mathit{cw}}

\DeclareMathOperator*{\scwop}{\mathit{scw}}
\DeclareMathOperator*{\poly}{\mathit{poly}}

\newcommand{\tw}{tree-width}
\newcommand{\cw}{clique-width}
\newcommand{\nlc}{NLC-width}
\newcommand{\scw}{slick clique-width}
\newcommand{\kex}{$k$-expression}
\newcommand{\kexs}{$k$-expressions}
\newcommand{\skex}{slick $k$-expression}
\newcommand{\skexs}{slick $k$-expressions}

\newcommand{\nlckexs}{NLC $k$-expressions}

\newcommand{\laa}[1]{\mathlarger {\mathlarger #1}}

\newcommand{\et}[2]{\eta_{#1,#2}}
\newcommand{\rh}[2]{\rho_{#1 \rightarrow #2}}

\newtheorem{theorem}{Theorem}
\newtheorem{lemma}{Lemma}
\newtheorem{corollary}{Corollary}

\newtheorem{remark}{Remark}

\newtheorem{example}{Construction}

\begin{document}
\pagestyle{myheadings}

\title{Eigenvalue location in graphs of small clique-width}
\author{Martin F\"{u}rer}
\address{Dept. of Computer Science and Engineering, Pennsylvania State University}
\email{\tt furer@cse.psu.edu}
\author{Carlos Hoppen}
\address{Instituto de Matem\'atica, Universidade Federal do Rio Grande do Sul}
\email{\tt choppen@ufrgs.br}
\author{David P. Jacobs}
%\address{ School of Computing, Clemson University Clemson, SC 29634, USA}
\address{ School of Computing, Clemson University}
\email{\tt dpj@clemson.edu}
\author{Vilmar Trevisan}
\address{Instituto de Matem\'atica, Universidade Federal do Rio Grande do Sul}
\email{\tt trevisan@mat.ufrgs.br}
\pdfpagewidth 8.5 in
\pdfpageheight 11 in

\begin{abstract}
Finding a diagonal matrix congruent to $A - cI$ for constants $c$, where $A$
is the adjacency matrix of a graph $G$ allows us to quickly tell
the number of eigenvalues in a given interval. If $G$ has clique-width $k$ and a corresponding $k$-expression is known,
then diagonalization can be done in time $O(\poly(k) n)$ where $n$ is the order of $G$.
\end{abstract}

\thanks{Research supported in part by NSF Grant CCF-1320814.}
\thanks{Work supported by Science without Borders CNPq - Grant 400122/2014-6, Brazil}

\maketitle


\section{Introduction}
\label{sec:intro}
Throughout this paper we use standard terminology for graph theory and linear algebra.
The main concern of spectral graph theory is to determine properties of a graph through the eigenvalues of matrices associated with it. Even if we restrict ourselves to the adjacency and the Laplacian matrices, eigenvalues and eigenvectors have been particularly useful for isomorphism testing and embedding graphs in the plane~\cite{Babai:1982}, for graph partitioning and clustering~\cite{Luxburg:2007}, in the study of random walks on graphs~\cite{Chung:1997,Lovasz1996} and in the geometric description of data sets~\cite{Coifman20065}, just to mention a few examples. An obvious step in any such application is to calculate the spectrum of the input graph, or at least to accurately estimate a subset of its eigenvalues. In fact, the distribution of eigenvalues of graphs in a given class of graphs generated by a given random graph model has been studied intensively (see~\cite{dumitriu2012,JiangDing2010,MCKAY1981203} and the references therein).

We say that an algorithm \emph{locates eigenvalues for a class $\mathcal{C}$} if, for any graph $G \in \mathcal{C}$ and any real interval $I$, it finds the number of eigenvalues of $G$ in the interval $I$.
In recent years, efficient algorithms have been developed for the location of eigenvalues
in trees \cite{JT2011},
threshold graphs \cite{JTT2013}
(also called {\em nested split graphs}),
and chain graphs \cite{MR3506491}.
A rich class of graphs which contain threshold graphs are the graphs with no induced subgraph isomorphic to $P_4$, which are often called {\em $P_4$-free} graphs or {\em cographs}.
Eigenvalue location in cographs and threshold graphs has been widely
studied \cite{Royle2003,Sander2008,Stanic2009,BSS2011,JacobsTT2015,MohammadianT2016}.

%%DISCUSS PURPOSE
It turns out that there is a strong connection between eigenvalue location and congruence of matrices, which we now describe.
Two matrices $R$ and $S$ are {\em congruent},
which we write $R \cong S$,
if there exists a nonsingular matrix $P$ for which $R = P^T S P$.
Let $G$ be a graph with adjacency matrix $A$, and
consider real numbers $c<d$.
If we can construct a diagonal matrix $D_c \cong B = A - cI$, then Sylvester's Law of Inertia \cite[p. 568]{Meyer2000} implies that the number $n_1$ of eigenvalues
of $A$ greater than $c$
equals the number positive entries in $D_c$. (Similarly, the number of eigenvalues equal to $c$, or less than $c$, are given by the number of zero diagonal entries, or by the number of negative entries in $D_c$, respectively.)
Hence, the number $n_2$ of positive entries in a diagonal matrix $D_d \cong A - dI$ is the number of eigenvalues of $A$ greater than $d$.
Thus $n_1 - n_2$ is the number of eigenvalues in $(c,d]$. This is why we want to design a fast algorithm to find a diagonal matrix that is congruent to $A - cI$.

Tura and two of the current authors~\cite{JacobsTT2015} designed such a diagonalization algorithm for cographs. This algorithm runs in linear time and performs congruence operations on the matrix $A-cI$ using the \emph{cotree} representation of $G$ (see~\cite{Corneil1981} for more information about cographs and the terminology associated with them). The algorithm works bottom-up on the cotree, and at each stage it diagonalizes the rows and columns associated with either one or two vertices (so-called \emph{siblings}), and then removes the corresponding leaves from the tree. Here, we generalize this approach to arbitrary graphs, using a parse tree representation that is closely connected to the hierarchical decomposition of graphs known as \emph{clique-width}. In spite of this similarity, the new algorithm requires several new ingredients. Indeed, unlike the cograph algorithm, the new algorithm does not diagonalize a given number of vertices at each stage, and needs to pass information up the tree (in a very compact way).

Clique-width is a powerful concept which was introduced in 2000 by Courcelle and Olariu \cite{CourcelleO2000}, and turns out to be interesting for algorithmic purposes. Its main motivation was to extend the well-known concept of tree-width due to Robertson and Seymour~\cite{RobertsonS86} to denser graphs. In general, graph widths have been used to design algorithms for NP-complete or even harder problems that are efficient on graphs of bounded width. (Interested readers are referred to~\cite{Bodlaender:2008,Niedermeier2006}, and the references therein. See also~\cite{Corneil:2005} for relations between tree-width and clique-width.)

A \emph{$k$-expression} is an expression formed from atoms  $i(v)$, two unary operations $\et{i}{j}$ and $\rh{i}{j}$, and a binary operation $\oplus$ as follows.
\begin{itemize}
\item $i(v)$ creates a vertex $v$ with label $i$, where $i$ is  from the set $[k]=\{1, \dots , k\}$.
\item $\et{i}{j}$ creates edges (not already present) between every vertex with label $i$ and
every vertex with label $j$ for $i \neq j$.
\item $\rh{i}{j}$ changes all labels $i$ to $j$.
\item $\oplus$ produces the disjoint union of two labeled graphs.
\end{itemize}
Finally, the graph generated by a $k$-expression is obtained by deleting the labels.
The \emph{clique-width} $\cwop(G)$ of a graph $G$ is the smallest $k$ such that the graph can
be defined by a $k$-expression \cite{CourcelleER93,CourcelleO2000}.

Any graph can be constructed in this way, provided that $k$ is large enough. For instance, cographs are exactly the graphs for which $\cwop(G) \leq 2$, and
one can show that
$\cwop(T) \leq 3$
for any tree $T$.
See \cite{KaminskiLM2009} for a discussion of the \cw\ of many classical classes of graphs.
Computing the clique-width is NP-hard \cite{FellowsRRS2006}. Thus, one usually assumes that a graph is given together with a $k$-expression.

%%%%%%%%PURPOSE
The purpose of this paper is to give an $O(\poly(k) n)$ time diagonalization algorithm for
graphs having \cw\ $k$. Note that the adjacency matrices of graphs with \cw\ $k$ often have $\Omega(n^2)$ nonzero entries, and that the \cw\ may be a small constant even if other parameters, such as the \tw, are linear in~$n$.
While there is a strong connection between tree-width and Gaussian elimination, the main
application area for graph widths has been the design of efficient algorithms
for NP-complete or even harder problems. The goal there is to find problems fixed parameter tractable (FPT), by providing an algorithm with a
running time of $O(f(k) n^c)$, for a constant $c$ and an arbitrary computable function $f$. Typically, $f$ is at least exponential, but, for small values of $k$, such algorithms are often very practical.

Here, we return to a polynomial time solvable problem. Nevertheless, the parameterized complexity
view is very useful. For bounded \cw, we turn a cubic time solution into a linear time solution,
a drastic improvement for graphs of small \cw. Despite being inspired by the cograph approach of~\cite{JacobsTT2015}, the extension is not at all straightforward. In fact, a property that is particularly crucial to us is that subgraphs generated by subexpressions are induced subgraphs, which does not hold for $k$-expressions, where an edge creating operation applied after a join
of $G$ and $H$ typically introduces edges within $G$ and within $H$. To deal with this, we introduce a new graph decomposition that may be easily obtained from a $k$-expression (and translated back into a $k$-expression).

Our approach is reminiscent of a parameter closely related to {\cw}, the lesser known {\em \nlc}, due to Wanke \cite{Wanke94},
and initiated by node label controlled (NLC)
graph grammars \cite{JanssensR80a,JanssensR80b}.
Graphs of \nlc\ at most $k$ are defined by \nlckexs.
These expressions
contain the operators
$i(v)$ and
$\rh{i}{j}$
for vertex creation and relabeling.
But new edges are created in {\em combination} with the join operation,
using a binary operation $\oplus_S$, where $S \subseteq [k] \times [k]$.
When $G \oplus_S H$ is applied,
then, for each $(i,j) \in S$, edges are introduced between
vertices labeled $i$ in $G$ and vertices labeled $j$ in $H$.
This has the effect that a subgraph generated by a subexpression is always an induced subgraph,
a property important to us.
%This is different from \kexs\ defining \cw, where an edge creating operation applied after a join of $G$ and $H$ typically introduces edges within $G$ and within $H$.

In representing graphs, we will actually use a minor but simplifying modification of \nlc, which we call {\em \scw} and is much more convenient than clique-width for our purposes.
Here a single operator performs the join, edge creation and relabelling, resulting in an expression whose parse tree is the simplest among graph representations.
It also has the property that subgraphs generated by subexpressions are induced subgraphs.

%%% DESCRIBE ORGANIZATION OF PAPER
The remainder of this paper is organized as follows.
In Section~\ref{sec:slickcw}
we define {\skexs} and {\scw}, denoted $\scwop(G)$.
We also prove that
$\scwop(G) \leq  \cwop(G) \leq 2 \scwop(G)$
%% $\scwop(G) \leq  \cwop(G) \leq 2 \scwop(G)$
and observe that there are linear-time transformations
to translate a {\kex} to a {\skex},
and a {\skex} to a $2k$-expression.
In Section~\ref{sec:algorithm} we describe
our $O(\poly(k) n)$ time diagonalization algorithm for graphs of {\scw} $k$, and in Section~\ref{sec:example} illustrate the algorithm by computing the inertia of a certain graph of order seven.
Some applications of this algorithm can be found in Section~\ref{sec:appl}, while
concluding remarks appear in Section~\ref{sec:conclusion}.

\section{Slick Clique-Width}
\label{sec:slickcw}

In the following definition a single operator is used for performing the union, creating edges
and relabeling.
\begin{samepage}
A \emph{\skex} is an expression formed from atoms $i(v)$ and a binary operation $\oplus_{S,L,R}$, where $L$, $R$ are functions from $[k]$ to $[k]$ and $S$ is a binary relation on $[k]$, as follows.
\begin{itemize}
\item[(a)] $i(v)$ creates a vertex $v$ with label $i$, where $i \in [k]$.
\item[(b)] Given two graphs $G$ and $H$ whose vertices have labels in $[k]$, the labeled graph $G \oplus_{S,L,R} H$ is obtained by the following operations.
Starting with the disjoint union of $G$ and $H$, add edges from every vertex labeled
$i$ in $G$ to every vertex labeled $j$ in $H$ for all $(i,j) \in S$.
Afterwards, every label $i$ of the \emph{left component} $G$ is replaced by $L(i)$, and every label $i$ of the \emph{right
component} $H$ is replaced by $R(i)$.
\end{itemize}
Two {\scw} expressions are said to be \emph{equivalent} if they produce the same labeled graph. Finally, the \emph{graph generated by a {\scw} expression} is obtained by deleting the labels of the labeled graph produced by it.
\end{samepage}
The \emph{\scw}
$\scwop(G)$
of a graph $G$
is the smallest $k$ such that the graph can be defined by a \skex.

Note that cographs are precisely the graphs with {\scw} equal to one. Indeed, recall that $G$ is a cograph if and only if either $G$ is a single vertex
or the union $G_1 \cup G_2$ or join $G_1 \join G_2$ of cographs $G_1$ and $G_2$ (see \cite{BSS2011}). On the other hand, when there is a single label available, the function $L$ and $R$ are trivial identities, so that $\oplus_{S,L,R}$ either creates a disjoint union (if $S = \emptyset$) or adds all possible edges with ends in the two operands (if $S=\{(1,1)\}$). The graph $G$ of Figure \ref{fig:int1} may be constructed with the following slick $2$-expression:
\begin{multline*}
  \left[1(f) \oplus_{\{(1,2)\},id,id}  2(g)\right]\oplus_{\{(1,2)\},\{2\rightarrow 1\},\{2\rightarrow 1\}} \\
  \left[\left\{1(e)\right\}  \oplus_{\{(1,2)\},id,\{ 1\rightarrow  2\}}   \left\{\left((1(c) \oplus_{\{(1,2)\},id,id} 2(d))\right)  \oplus_{\{(2,2)\},id,id} \left((1(a) \oplus_{\{(1,2)\},id,id}  2(b))\right)\right\}\right] \\
\end{multline*}
or, representing the triples that define each operation by letters,
$$
  \left[1(f) ~\textbf{B}~  2(g)\right]~ \textbf{A} ~
  \left[\left\{1(e)\right\} ~ \textbf{C} ~ \left\{\left((1(c)  ~\textbf{E}~ 2(d))\right) ~\textbf{D}~ \left((1(a) ~\textbf{F}~  2(b))\right)\right\}\right].
$$
Since $G$ is not a cograph (it contains an induced copy of $P_4$), we know that $\scwop(G)=2$.

\begin{figure}[h!]\begin{center}
\begin{tikzpicture}
  [scale=1.5,auto=left,every node/.style={circle,scale=.6}]
  \node[draw,circle] (1) at (0,0) {$a$};
  \node[draw,circle] (2) at (1,0) {$b$};
  \node[draw,circle] (3) at (1.5,1) {$e$};
  %\node[draw,circle,label=above right:4] (4) at (2.5,1) {};
  \node[draw,circle] (5) at (2,0) {$d$};
  \node[draw,circle] (6) at (3,0) {$c$};
  \node[draw,circle] (7) at (1.5,-1) {$g$};
  \node[draw,circle] (8) at (2.5,-1) {$f$};
  \foreach \from/\to in {1/2,2/3,2/5,3/5,5/6,7/8,7/1,7/2,7/5,7/6} {
    \draw (\from) -- (\to);}
\end{tikzpicture}
     \caption{Graph with $k=2$.}
     \label{fig:int1}
\end{center}\end{figure}
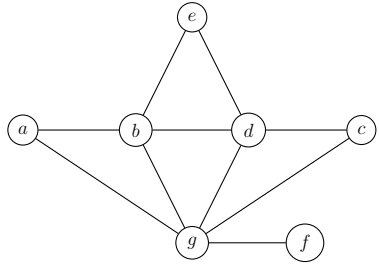

Note that the definition of a \skex{} implies that edges can only be placed between different components,
so if two vertices are in the same component after some steps in the above construction, but are not adjacent,
they will never become adjacent.
As shown in Figure~\ref{fig:parsetree},
a {\skex} can also be represented as a parse tree $T$
where the leaves contain the operators $i(v)$
and the internal nodes contain the $\oplus_{S,L,R}$ operations.
Two vertices $v$ and $w$ are adjacent if and only if
their least common ancestor $\oplus_{S,L,R}$ in $T$ connects them,
similar to the cotree representation for cographs in \cite{BSS2011}.
\begin{figure}[h!]
\begin{center}
\begin{tikzpicture}
  [scale=1,auto=left,every node/.style={circle,scale=0.9}]
  \node[draw, circle, fill=blue!10, inner sep=0,label=below:A] (A) at  (0,0) {$\oplus$};
  \node[draw, circle, fill=blue!10, inner sep=0,label=below:B] (B) at  (-1.5,-1) {$\oplus$};
  \node[draw, circle, fill=blue!10, inner sep=0,label=below:C] (C) at  (1.5,-1) {$\oplus$};
  \node[draw,circle,fill=black,label=below:$1(f)$,scale=.6] (1) at (-2,-2) {};
  \node[draw,circle,fill=black,label=below:$2(g)$,scale=.6] (2) at (-1,-2) {};
  \node[draw,circle,fill=black,label=below:$1(e)$,scale=.6] (3) at (1,-2) {};
  \node[draw, circle, fill=blue!10, inner sep=0,label=below:D] (D) at  (3,-2) {$\oplus$};
  \node[draw, circle, fill=blue!10, inner sep=0,label=below:E] (E) at  (1.5,-3) {$\oplus$};
  \node[draw, circle, fill=blue!10, inner sep=0,label=below:F] (F) at  (4.5,-3) {$\oplus$};
  \node[draw,circle,fill=black,label=below:$1(c)$,scale=.6] (4) at (1,-4) {};
  \node[draw,circle,fill=black,label=below:$2(d)$,scale=.6] (5) at (2,-4) {};
  \node[draw,circle,fill=black,label=below:$1(a)$,scale=.6] (6) at (4,-4) {};
  \node[draw,circle,fill=black,label=below:$2(b)$,scale=.6] (7) at (5,-4) {};
 \foreach \from/\to in {A/B,A/C,B/1,B/2,C/3,C/D,D/E,D/F,E/4,E/5,F/6,F/7} {
    \draw (\from) -- (\to);}
\end{tikzpicture}
       \caption{A parse tree.}
       \label{fig:parsetree}
\end{center}
\end{figure}
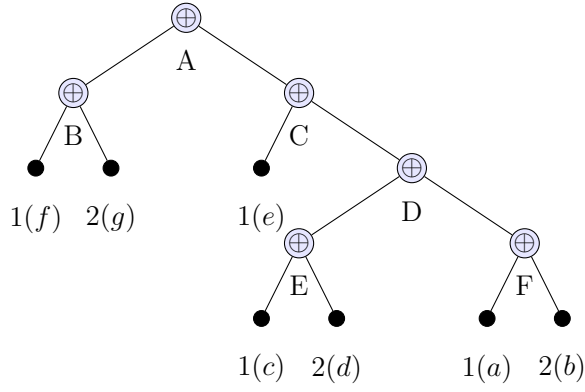

We can define the {\em depth} $d(r)$ of a {\skex } $r$ recursively.
The expression $i(v)$ has depth $0$,
and $d( A \oplus_{S,L,R} B) = 1 + \max \{ d(A), d(B) \} $.
This is equivalent to the depth of the parse tree for $r$.
In a similar way, we can define the depth of a {\kex}.

The next result shows that the concepts of {\cw} and {\scw} are closely related.
\begin{theorem}
\label{thr:scwcwratio}
If $G$ is a graph then  $\scwop(G) \leq  \cwop(G) \leq 2 \scwop(G)$.
\end{theorem}
\begin{proof}
This follows from Lemma~\ref{lem:scwop-cwop} and Lemma~\ref{lem:cwop-twoscwop} below.
\end{proof}

\begin{lemma}
\label{lem:eta-slick}
For any {\skex} $s$, we can write $\et{i}{j} (s)$ as a {\skex}
$s^\prime$ where $d(s) = d(s^\prime )$.
\end{lemma}
\begin{proof}
We apply induction on $d(s)$.
When $d(s) = 0$, then $s$ has the form $i'(v)$
and $\et{i}{j} (s)$ is equivalent to $i'(v)$.
Now let $d(s) > 0$, where $s = s_1 \oplus_{S, L ,R} s_2$,
for slick expressions $s_1$ and $s_2$,
and assume the statement holds for {\skexs} of smaller depth.
We must construct edges between all pairs of vertices
whose labels are $i$ and $j$ {\em after} $L$ and $R$ are applied.
The edges created by $\et{i}{j}$ are either {\em between}
left and right components or {\em within} a component.
To produce the edges
between components, let
$$
S^\prime =
\{ (i^\prime, j^\prime)  |   L(i^\prime) = i  \mbox{ and }  R(j^\prime) = j   \}
\cup
\{ (j^\prime, i^\prime)  |   L(j^\prime) = j  \mbox{ and }  R(i^\prime) = i    \}
$$
To produce edges within components let
\begin{eqnarray*}
S_{LL} = \{ (i^\prime, j^\prime)  |   L(i^\prime) = i  \mbox{ and }  L(j^\prime) = j   \} \\
S_{RR} = \{ (i^\prime, j^\prime)  |   R(i^\prime) = i  \mbox{ and }  R(j^\prime) = j   \}
\end{eqnarray*}
If $S_{LL}$ and $S_{RR}$ are both empty, then
$\et{i}{j}$ must not place any new edges within the components.
Hence we can complete the induction with the slick expression
\begin{equation}
\label{eq:lemma1a}
s_1 \oplus_{S \cup S^\prime, L ,R}  s_2 .
\end{equation}
However, suppose
$$
S_{LL} = \{ (i_1^\prime, j_1^\prime),  (i_2^\prime, j_2^\prime), \ldots,  (i_k^\prime, j_k^\prime)  \}
$$
is not empty.
To place edges into the left side
of (\ref{eq:lemma1a})
we construct the composition
\begin{equation}
\label{eq:lemma1}
\et{ i_{1}^\prime }{ j_{1}^\prime} ~~ \et{ i_{2}^\prime }{ j_{2}^\prime} ~~ \ldots \et{ i_{k}^\prime }{ j_{k}^\prime} ~~  s_1
\end{equation}
Since $d(s_1) < d(s)$, by repeated use of the induction assumption,
the expression (\ref{eq:lemma1}) can be written
as a single {\scw} expression
$s_1^\prime$
where $d( s_1^\prime) = d(s_1)$.
A similar argument shows that
if $S_{RR} \not= \emptyset$,
correct edges can be
placed into into right side of (\ref{eq:lemma1a})
as a single {\scw} expression $s_2^\prime$,
where $d( s_2^\prime) = d(s_2)$.
Hence we can write
$\et{i}{j} (s)$
as
$$
s_1^\prime \oplus_{S \cup S^\prime, L ,R}  s_2^\prime ,
$$
whose depth is that of $s$.
This completes the induction.
\end{proof}

\begin{lemma}
\label{lem:scwop-cwop}
If $G$ is a graph then  $\scwop(G) \leq  \cwop(G)$.
\end{lemma}
\begin{proof}
Consider a {\kex} $r$ composed of the operators $i(v)$, $\et{i}{j}$, $\rh{i}{j}$ and $\oplus$.
It suffices to construct an
{\em equivalent}
{\skex}, that is, one that produces the same labelled graph.
We use induction on $d(r)$, the base case being clear.
Let $d(r) > 0$,
and assume all {\kexs} of smaller depths have equivalent {\skexs}.
There are three cases.
Either $r = \rh{i}{j} (r_1)$, $r = r_1 \oplus r_2$, or $r = \et{i}{j} (r_1)$,
where the $r_i$ are {\kexs}.

If $r = \rh{i}{j} (r_1)$,
then by induction $r_1$ has an equivalent {\skex} $s_1$.
If $s_1$ has the form
$i(v)$ we may rewrite $r$ as the
{\skex} $j(v)$.
If $s_1 = a \oplus_{S, L ,R} b$,
we define the functions
$$
\begin{array}{ll}
    L^\prime(t)
     =
    \displaystyle{
            \left\{
                   \begin{array}{cl}
                      L(t)  &  \mbox{ if } L(t) \neq i \\
                        j                  &  \mbox{ otherwise }
                   \end{array}
           \right.}
&
    R^\prime(t)
     =
    \displaystyle{
            \left\{
                   \begin{array}{cl}
                      R(t)  &  \mbox{ if } R(t) \neq i \\
                        j                  &  \mbox{ otherwise }
                   \end{array}
           \right.}
\end{array}
$$
Then $r$ is equivalent to the {\skex} $a \oplus_{S, L^\prime ,R^\prime} b$.

Suppose $r = r_1 \oplus r_2$.
Then by induction, there are {\skexs}
$s_1$ and $s_2$ equivalent to $r_1$ and $r_2$ respectively.
Then $r$ must be equivalent to the
{\skex}
$s_1 \oplus_{\emptyset,\id ,\id} s_2$.

Finally, assume $r = \et{i}{j} (r_1)$.
By induction $r_1$ is equivalent to a {\skex} $s$,
so $r$ is equivalent to $\et{i}{j} (s)$.
By Lemma~\ref{lem:eta-slick}, this is equivalent to some {\skex},
and we are done.
\end{proof}

We note that Lemma~\ref{lem:scwop-cwop} (with Lemma~\ref{lem:eta-slick}) gives an algorithm
for translating a {\kex} to a {\skex}, and in fact is linear-time for constant $k$.
Similarly, Lemma~\ref{lem:cwop-twoscwop} below provides a linear-time algorithm for translating
a {\skex} to a $2k$-expression.
\begin{lemma}
\label{lem:cwop-twoscwop}
If $G$ is a graph then  $\cwop(G) \leq  2 \scwop(G)$.
\end{lemma}
\begin{proof}
Suppose $k = \scwop(G)$, and $s$
is a {\skex} for $G$.
It suffices to construct an equivalent $2k$-expression $r$. We will show this by induction on the depth of $s$.
This is clear when $s = i(v)$.
Otherwise, let $s = s_1  \oplus_{S,L,R} s_2$,
for
{\skexs} $s_1$ and $s_2$.
By the induction hypothesis, we can assume that each {\skex} $s_i$
has an equivalent $2k$-expression $r_i$,
producing the same labeled graph.
So we have
$$
s = r_1 \oplus_{S,L,R} r_2 .
$$
Note that, even though $r_1$ and $r_2$ may involve labels between $k+1$ and $2k$, they produce labeled graphs with labels in $[k]$. To complete the induction we translate the behavior
of $\oplus_{S,L,R}$ into the operators $\rh{i}{j}$, $\et{i}{j}$ and $\oplus$.
We relabel the vertices on the left side
mapping each label $i \in [k]$ to $i + k$,
and then form the union. Suppressing parentheses, we obtain a subexpression
\begin{equation}
\label{eq:cw-scw}
w = (\rh{1}{ k+1 } ~~  \rh{2}{ k+2 } ~~ \ldots ~~  \rh{k}{2k} ~~ r_1 ) \oplus r_2.
\end{equation}
To obtain edges,
for each $(i,j) \in S$ we apply the operators $\et{i+k}{j}$ to $w$,
obtaining a new expression $w^\prime$.
The relabeling
in (\ref{eq:cw-scw}) ensures that new edges are placed
{\em only} between the left and right sides.
Depending on the functions $L$ and $R$,
we finally relabel the left and right sides back to their desired values in $[k]$.

We need to relabel according to the function $f: [2k] \rightarrow [k]$ defined by
\[
f(i) = \begin{cases}
                L(i-k)    & \mbox{ if }  i > k\\
                 R(i)     & \mbox{ if } i \leq k.
	\end{cases}
\]
It is tempting to simply apply all the relabelings $\rh{i}{f(i)}$,
but one operation could modify earlier operations.
Instead, we do this relabeling in three rounds.

First, we reduce the number of labels to at most $k$, using the $2k$-expression
\[
w'' = \rh{1}{g(1)} \dots \rh{2k}{g(2k)} w'
\]
where
\[ g(i) = \max \{j \mid f(i)  = f(j) \}. \]
Note that the set $\{j \mid f(i)  = f(j) \}$ is not empty, as it contains the label $i$.
Also note that the order of operations and
$i \leq g(i)$
ensure that no $\rh{i}{g(i)}$ changes an earlier one.

Now, we move the set of current labels
$I = \{i_1,\dots,i_q\}$ with $i_1 < i_2 < \dots < i_q$ and $q \leq k$
up into the interval $[2k-q+1,2k]$ using the $2k$-expression
\[
w''' = \rh{i_1}{h(i_1)} \dots  \rh{i_q}{h(i_q)} w''
\]
where $h: I \rightarrow [2k-q+1,2k]$ is given by
$h(i_j) = 2k - q + j$. Since $i_j \leq h(i_j)$,
no operation can affect earlier ones. Note that the break down into the previous two rounds is just to simplify the description.
These relabelings could have been combined into one round.

In the third round, we choose the proper new labels with the $2k$-expression
\[
w'''' = \rh{2k-q+1}{f(h^{-1}(2k-q+1))} \, \rh{2k-q+2}{f(h^{-1}(2k-q+2))} \dots  \rh{2k}{f(h^{-1}(2k))} \, w'''
\]
As each operation maps a label in $[ k+1,2k]$ to one in $[k]$,
no  relabeling can injure others.
After the third round, $f$ is achieved.
Finally, to obtain $r$, we delete all the $\rh{i}{i}$ operations, or declare them as having no effect.
\end{proof}

The cograph diagonalization algorithm in \cite{JacobsTT2015}
exploited the fact that in any cograph of order $n \geq 2$,
there exist two vertices $u$ and $v$
for which either $N(u) = N(v)$ or $N[u] = N[v]$,
so-called siblings.
This means that their corresponding rows and columns in the adjacency matrix
can differ by at most two positions.
By subtracting say, the row (column) of $u$ from the row (column) of $v$,
the row and column of $v$ is annihilated except in one off diagonal position.
The following analog is crucial to our algorithm.

\begin{remark}
\label{obs:equalrows}
Let $T_G$ be a parse tree for a graph $G$ with adjacency matrix $A$,
and  $Q$ a node in $T_G$.
If two vertices $u$ and $v$ have the same label at $Q$,
then their rows (columns) will agree
outside of the matrix for the subtree rooted at $Q$.
\end{remark}

Two matrices are congruent if one can obtain
the other by a sequence of {\em pairs} of elementary operations,
each pair consisting of a row operation followed by the {\em same} column operation.
In our algorithm we only use congruence operations that permute rows and columns or add a multiple of a row and column to
another row and column respectively.
To achieve linear-time we must operate on a sparse representation
of the graph, rather than the adjacency matrix. Moreover, we must represent the undiagonalized portion of the matrix with constant space.

\section{The Algorithm}
\label{sec:algorithm}
We now describe our diagonalization algorithm.
Let $G = (V,E)$ be a graph on $n = |V|$ vertices,
and adjacency matrix $A$, given by a \skex\ $Q_G$.
We wish to find a diagonal matrix congruent to $B = A - c I_n$.
The expression $Q_G$ defines its parse tree $T$ having $2n - 1$ nodes,
a rooted binary tree whose nodes are the subexpressions
of $Q_G$, and whose edges are the pairs of nodes
$\{Q_\ell,Q\}$ and $\{Q_r,Q\}$ for some subexpression $Q$ of $Q_G$ with
$Q = Q_\ell \oplus_{L,R,S} Q_r$ for some $L$, $R$, and $S$.
The algorithm {\scwalgo}
works bottom-up in the parse tree $T$ of the \skex\ $Q_G$.
In particular, our implementation does a
post-order traversal and operates on $Q$ after both children have been processed.

A node $Q$ in the tree produces a data structure that we call a \emph{$k$-box} $b_Q$, namely a $4$-tuple $[k',k'',M,\Lambda]$ where $k'$ and $k''$ are nonnegative integers bounded above by $k$, $M$ is a symmetric matrix of order $m \leq 2k$ and $\Lambda$ is a vector of whose $m$ components are labels in $\{1,\ldots,k\}$. In a high level description, the algorithm traverses the parse tree from the leaves to the root so that, at each node of the parse tree, the algorithm either initializes a box, or it combines the boxes produced by the node's children into its own box, transmitting it to its parent. While processing the node, the algorithm may also produce diagonal elements of a matrix congruent to $A-cI_n$. These diagonal elements are appended to a global array as they are produced.
\begin{figure}[h]\label{fig:highlevel}
{\tt
\begin{tabbing}
aaa\=aaa\=aaa\=aaa\=aaa\=aaa\=aaa\=aaa\= \kill
     \> \texttt{diagonalize}(G,c)\\
     \> {\bf input}: the parse tree $T$ of slick $k$-expression $Q_G$ for $G$, a scalar $c$ \\
     \> {\bf output}: diagonal entries in $D \cong B = A(G) - c I$ \\
     \> Order the $2n-1$ vertices of $T$ as $Q_1,Q_2,\ldots,Q_{2n-1}=Q_G$ in post order\\
     \>{\bf for} $t$ {\bf from } 1 \bf {to} $2n-1$ {\bf do} \\
     \> \> {\bf if} is-leaf($Q_t$) {\bf then construct} $b_{Q_t}$=\texttt{LeafBox}($Q_t,c$)\\
     \> \> {\bf else if} $Q_t = Q_\ell \bigoplus_{S,L,R} Q_r$ \\
     \> \> \> {\bf then construct} $b_{Q_t}$=\texttt{CombineBoxes}($b_{Q_\ell},b_{Q_r}$)\\
     \> \texttt{DiagonalizeBox}($b_{Q_{2n-1}}$)
\end{tabbing}}
\caption{High level description of the algorithm \texttt{diagonalize}.}
\end{figure}

In the remainder of this section, we shall describe each of these stages in detail and prove that the algorithm \texttt{diagonalize} yields the desired output.

At each node $Q$, our algorithm operates on a small $O(k) \times O(k)$ matrix,
performing congruence operations.
These operations represent operations that would be performed on the large $n \times n$ matrix $B$.
%This small matrix represents a {\em partial view} of the actual $n \times n$ matrix.
%The algorithm outputs the sequence of diagonal elements as it computes them.

Recall that any subexpression $Q$ of $Q_G$
is also a node in the parse tree $T$ of $Q_G$.
For a node $Q$ of $T$, let $n_Q$ be the order of the graph $G(Q)$ generated by
the expression $Q$.
Note that $G(Q)$ is the subgraph of $G$ induced by the vertices of $G(Q)$.
Let $A_Q$ be the adjacency matrix of $G(Q)$, let $I_Q$ be the $n_Q \times n_Q$ identity matrix,
and $B_Q = A_Q - cI_Q$. For simplicity, an entry $vw$ of such a matrix $M$ will always refer to the vertices of the graph $G$ indexing the corresponding rows in $M$. (This avoids the need of keeping track of permutations of rows and columns when describing the algorithm.)

The goal at each node $Q_t$ is to construct, by means of congruence operations, a matrix $\submat{Q}_t$
which is diagonal except for at most $2k$ rows and columns, having the form:
\begin{equation}
\label{eq:matrixBQ}
\submat{Q}_t =
 \left( \,
 \begin{array}{cccccc}
d_1& & \mbox{\large $0$} &  \\
        & \ddots && \multicolumn{3}{c}{\mbox{\LARGE $0$}}  \\
\mbox{\large $0$} & & d_{\ell} \\
 \multicolumn{3}{c}{\mbox{\LARGE $0$}} &
\multicolumn{3}{c}{
\begin{array}{|c|cc|}\hline
& & \\
M^{(0)} & \hspace*{3mm} & M^{(1)} \hspace*{5mm}\\
& & \\\hline
& & \\
& & \\
M^{(1)T} & \hspace*{3mm} & M^{(2)} \hspace*{5mm}\\ & & \\
& & \\\hline
\end{array}
}
\end{array} \;
\right).
\end{equation}
Here, $M^{(0)}$ and $M^{(2)}$ are square matrices of dimensions $k' \times k'$ and $k'' \times k''$
respectively, $M^{(1)}$ is a $k' \times k''$ matrix with $0 \leq k' \leq k'' \leq k$.
Note that $k'$ can be zero in which case we regard $M^{(0)}$ as empty.
The criterion that defines the partition of the matrix
\begin{equation}
\label{eq:formofM}
\M{Q_t} = \begin{array}{|c|cc|}\hline
& & \\
M^{(0)} & \hspace*{3mm} & M^{(1)} \hspace*{5mm}\\
& & \\\hline
& & \\
& & \\
M^{(1)T} & \hspace*{3mm} & M^{(2)} \hspace*{5mm}\\ & & \\
& & \\\hline
\end{array}
\end{equation}
may be visualized in equation (\ref{eq:bigmatrixBQ}).

As it turns out, the matrix $\submat{Q}_t$ will be congruent to $B_{Q_t}$. Moreover, it is crucial to control its relationship to the matrix $\bigmat{Q}_t$.
This is the $n \times n$ matrix that {\em would be} obtained by performing
the same congruence operations on $B$ that are actually performed up to stage $t$, that is, on $\submat{Q}_1, \submat{Q}_2, \ldots, \submat{Q}_t$.
%%% IS THIS OK?
Thus the invariant $B \cong \bigmat{Q}_t$ is maintained. Equation~\eqref{eq:bigmatrixBQ} illustrates how the matrix $\submat{Q}_t$ fits into $\bigmat{Q}_t$, which is never actually computed by the algorithm. We assume, for clarity, that the rows and columns corresponding to vertices that have already been diagonalized or lie in $G(Q_t)$ appear first.
\begin{equation}
\label{eq:bigmatrixBQ}
\bigmat{Q}_t =
 \left( \,
 \begin{array}{ccccccccc}
\multicolumn{3}{c}{\mbox{\LARGE $D$}}
        &  \multicolumn{3}{c}{\mbox{\LARGE $0$}} & \multicolumn{3}{c}{\mbox{\LARGE $0$}} \\[5mm]
\multicolumn{3}{c}{\mbox{\LARGE $0$}} &
\multicolumn{3}{c}{
\begin{array}{|c|cc|}\hline
& & \\
M^{(0)} & \hspace*{3mm} & M^{(1)} \hspace*{5mm}\\
& & \\\hline
& & \\
& & \\
M^{(1)T} & \hspace*{3mm} & M^{(2)} \hspace*{5mm}\\ & & \\
& & \\\hline
\end{array}
}&
\substack{ 0 \\ \\ \vdots \\ \\ 0 \\  \beta_1^1 \\ \\ \\ \vdots
\\ \\ \\ \\ \beta^1_{k^{''}}} &
\substack {\cdots \\  \\ \\ \\ \\ \cdots \\  \cdots \\ \\ \\ \\ \\
\\ \\ \\ \\ \cdots }&
\substack{0 \\ \\ \vdots \\ \\ 0 \\  \beta_1^s \\ \\ \\ \vdots
\\ \\ \\ \\ \beta_{k^{''}}^s }\\
&&&&&&&&\\
\multicolumn{3}{c}{\mbox{\LARGE $0$}}  &
\substack{\hspace{2mm} 0 \hspace{4mm} \cdots \\ \\ \hspace*{-4mm} \vdots\\ \\ \\ \hspace{4mm} 0 \hspace{6mm} \cdots} &
\substack{0 \; \beta_1^1 \\ \\\hspace{-2mm}\vdots \; \;\vdots\\ \\ \\0\;\beta_1^s }
\hspace{5mm}\substack{ \cdots \\ \\ \\ \\ \\ \\\cdots} & \hspace{5mm}\substack{\beta^1_{k^{''}} \\ \\ \vdots \\ \\ \\ \beta^s_{k^{''}}} &  & \laa{M^\prime}\\
\end{array} \;
\right)
\end{equation}
The right side of (\ref{eq:bigmatrixBQ}) shows how the remainder of the large matrix
$B$ is transformed as a side effect of operating on the small submatrix $\M{Q_t}$.
The diagonal matrix $D$ represents {\em all} diagonalized elements produced up until
stage $t$ in the algorithm.
Also the $k'$ rows and columns of $M^{(0)}$ extend with zero vectors, defining the
boundary between $M^{(0)}$ and $M^{(2)}$.
The $\beta_i^j$ are zero-one entries in the partially diagonalized matrix, whose relation with the corresponding entries in the original matrix $B$ will be explained later (see Lemma~\ref{lem:description}).
It is important to observe that after node $Q_t$ has been processed,
all vertices in the subgraph $G(Q_t)$ correspond to rows in $D$ or $\M{Q_t}$.
Some rows of $D$ may correspond to vertices outside of $G(Q_t)$, which have been diagonalized in earlier stages. The submatrix  $M^\prime$ in (\ref{eq:bigmatrixBQ}) contains all undiagonalized rows  $w \not\in G(Q_t)$, may include vertices in $\M{Q_{t'}}$ for $t' \neq t$, and is empty after the last iteration of the algorithm.

It will be convenient to define the $k'$ rows in $M^{(0)}$ as having {\em \typei},
and to define the $k''$ rows of
$M^{(2)}$ as having {\em \typeii}.
It is useful to understand that a row begins as a {\typeii} row, then becomes a {\typei} row, and finally becomes diagonalized.
%Moreover, it will always be possible to make $M^{(0)}$ zero or empty.

As mentioned earlier, diagonal elements in diagonalized rows and columns are not transmitted by nodes
to their parents. But label information of vertices corresponding to the rows in $\M{Q_t}$ must be maintained,
and for simplicity, we will say that rows have labels.
The important information is $M=\M{Q_t}$, these labels, and the integers $k'$ and $k''$,
which are stored in the $k$-box (or simply box) $b_{Q_t}=[k',k'',M,\Lambda]$ mentioned above,
whose {\em size} is $k' + k''$.
To ensure that $M^{(2)}$, whose rows have {\typeii}, has order $k'' \leq k$ in the matrix transmitted,
each label appears in at most one \typeii{} row.

%Each node $Q$ creates a box of size at most $2k$ whose matrix $\M{Q}$ contains the undiagonalized portion of $\submat{Q}$.
When the node $Q_t$ is a leaf corresponding to a subexpression $i(v)$, $B_{Q_t} = [-c]$.
Therefore the box contains a $1 \times 1$ matrix $\M{Q_t} = [-c]$,
whose row is labeled $i$, $k' = 0$, and $k'' = 1$.
\begin{figure}[h]\label{fig:leafbox}
{\tt
\begin{tabbing}
aaa\=aaa\=aaa\=aaa\=aaa\=aaa\=aaa\=aaa\= \kill
     \> \texttt{LeafBox}($Q$,$c$)\\
     \> {\bf input}: $k$-expression $Q=i(v)$, a scalar $c$ \\
     \> {\bf output}: $\left[0,1,[-c],[i]\right]$ \\
\end{tabbing}}
\caption{Procedure \texttt{LeafBox}.}
\end{figure}

To conclude the description of the algorithm \texttt{diagonalize}, we need to define the procedure \texttt{CombineBoxes}. Before describing this process, we state a lemma that summarizes facts about the algorithm that help establish its correctness. We observe that the rows (and columns) in the matrix $\submat{Q}_t$ that are not represented in $\M{Q_t}$ have been diagonalized earlier in the process, so that $\submat{Q}_t$ contains
three types of rows: \typei{}, \typeii{} and diagonalized. For simplicity, when referring to operations, we always mention the row operations, with the understanding that the corresponding column operations are also performed. We also identify vertices with their rows and columns.
\begin{lemma} \label{lem:description}
Let $Q_1,Q_2,\ldots, Q_{2n-1}=Q_G$ be the nodes of the parse tree of the $k$-expression $Q_G$ that defines $G$, listed in post order. Consider the matrices $B$, $\M{Q_{\tau}}$, $\submat{Q}_{\tau}$ and $\bigmat{Q}_{\tau}$ defined above, where $\tau \in \{1,\ldots,2n-1\}$, and let $v$ and $w$ be vertices of $G$.
\begin{itemize}
\item[(a)] If row $v$ is diagonalized at stage $\tau$, then stage $\tau$ applied CombineBoxes. Moreover, if row $v$ was diagonalized when processing $Q_{\tau'}$ for some $1\leq \tau'<\tau$, then no components in this row are modified in stage $\tau$.
\item[(b)] Suppose that  row $v$ has \typei{} in $\M{Q_{\tau'}}$, for some $\tau' \in \{1,\ldots,\tau\}$. Suppose also that $w$ is a row such that, for all $j \in \{\tau',\ldots,\tau\}$,  $v$ and $w$ are not simultaneously in $G(Q_{j})$. Then the entry $vw$ in $\bigmat{Q}_\tau$ is $0$.
\item[(c)] Suppose that row $v$ has \typeii{} in $\M{Q_\tau}$ and $w \not \in G(Q_{\tau})$.  Suppose also that, for all $\tau' < {\tau}$, $w$ has \typeii{} in $\submat{Q}_{\tau'}$ whenever $w \in G(Q_{\tau'})$.  Then $\bigmat{Q}_{\tau}$ and $B$ are equal in position $vw$.
\item[(d)] Suppose that, for some ${\tau}'\in \{0,\ldots,{\tau}-1\}$, the rows $v$ and $w$ are both not in $G(Q_j)$ for all ${\tau}'<j \leq {\tau}$. Then $\bigmat{Q}_{\tau}$ and $\bigmat{Q}_{\tau'}$ are equal in position $vw$ (where $\bigmat{Q}_{0}=B$).
\item[(e)] Suppose that row $v$ has \typei{} in $\M{Q_{\tau'}}$, for some $\tau' \in \{1,\ldots,\tau-1\}$. For all $\tau'<j \leq \tau$, if $v$ is in $\M{Q_{j}}$, then it has \typei{}. That is, a {\typei} vertex cannot become {\typeii} again.
\end{itemize}
\end{lemma}

Next we will show how \texttt{CombineBoxes} works, that is, we explain how a node produces its box from the boxes transmitted by its children. We shall assume inductively that Lemma~\ref{lem:description} holds for all $\tau$ up to step $t-1$. This can be done because all items hold trivially at the beginning of the algorithm and cannot cease to hold after processing a leaf node, as no congruence operations are performed. Suppose that $Q_t$ is a node with children $Q_{\ell}$ and $Q_r$, that is
$$
Q_t = Q_{\ell} \oplus_{S,L,R} Q_r .
$$
Let $\M{Q_\ell}$ and $\M{Q_r}$
denote the matrices in the boxes transmitted
respectively, by
$Q_{\ell}$ and $Q_r$,
containing the undiagonalized rows in $\submat{Q_\ell}$ and $\submat{Q_r}$.
The goal of $Q_t$ is to combine
$\M{Q_\ell}$ and $\M{Q_r}$
into a single matrix of size at most $2k$
representing $\submat{Q}_t$. More precisely, based on $\M{Q_\ell}$ and $\M{Q_r}$, node $Q_t$ first constructs the submatrix $M$ of $\bigmat{Q}_{t-1}$ induced by the undiagonalized rows in $\submat{Q_\ell}$ and $\submat{Q_r}$. Recall that this cannot be done directly, as $\bigmat{Q}_{t-1}$ is never actually computed by the algorithm.

\begin{lemma}\label{lemma_bigmat}
Suppose that $Q_t= Q_{\ell} \oplus_{S,L,R} Q_r $ and consider the submatrix $M$ of $\bigmat{Q}_{t-1}$ induced by the undiagonalized rows in $\submat{Q_\ell}$ and $\submat{Q_r}$. The following hold:
\begin{itemize}
\item[(a)] If $s \in \{\ell,r\}$ and $v,w \in G(Q_s)$, then the entry $vw$ in $M$ is equal to the entry $vw$ in $\M{Q_s}$.

\item[(b)] If $v \in G(Q_\ell)$, $w \in G(Q_r)$, and at least one of $v$ and $w$ has \typei{}, then the entry $vw$ in $M$ is 0.

\item[(c)] If $v \in G(Q_\ell)$, $w \in G(Q_r)$, and both have \typeii{}, then the entry $vw$ in $M$ is $1$ if $(i,j) \in S$, where $i$ is the type of $v$ and $j$ is the type of $w$, otherwise it is 0.
\end{itemize}
\end{lemma}

\begin{proof}
Consider the entry $vw$ in a submatrix $M$ as in the statement of the lemma. By the definition of slick expression, $V(Q_\ell) \cap V(Q_r) = \emptyset$.

For part (a), assume that both $v$ and $w$ are on the same side, say $v,w \in \M{Q_\ell}$. Because the nodes of the parse tree are processed in post-order, $v$ and $w$ cannot be in $V(Q_j)$ for $\ell<j<t$. If $\ell<t-1$, Lemma~\ref{lem:description}(d) applied to $\tau'=\ell$ and $\tau=t-1$ implies that the entry $vw$ is the same in $\bigmat{Q}_{t-1}$ and $\bigmat{Q}_{\ell}$. If $\ell=t-1$, this equality is trivial. By the definition of $\M{Q_\ell}$ in~\eqref{eq:matrixBQ}, the entries $vw$ in $\bigmat{Q}_{\ell}$ and $\M{Q_\ell}$ are equal. The desired conclusion follows because $M$ is a submatrix of $\bigmat{Q}_{t-1}$. Clearly, the analogous statement holds for $v,w \in \M{Q_r}$ and the matrices $\bigmat{Q}_{t-1}$ and $\bigmat{Q}_{r}$.

For part (b), fix $v \in \M{Q_\ell}$ and $w \in \M{Q_r}$. First assume that $v$ has \typei{} in $\M{Q_\ell}$. Note that $w\notin \M{Q_\ell}$ and $v$ is not in $G(Q_j)$ for $j \in \{\ell+1,\ldots,t-1\}$. By Lemma~\ref{lem:description}(b) for $\tau'=\ell$ and $\tau=t-1$, the entry $vw$ in $\bigmat{Q}_{t-1}$ is equal to 0. The same conclusion would be achieved if we assumed that $w$ has \typei{} in $\M{Q_r}$.

We finally consider the case where $v$ and $w$ have  \typeii{} in $\M{Q_\ell}$ and $\M{Q_r}$, respectively. Without loss of generality, assume that $\ell>r$. We know that $w$ has \typeii{} in $\M{Q_r}$ and that $w \notin G(Q_j)$ for all $j \in \{r+1,\ldots,t-1\}$. Moreover, Lemma~\ref{lem:description}(e) implies that, if $j \in \{1,\ldots,r-1\}$ and $w \in G(Q_j)$, then $w$ has \typeii{}. We may apply Lemma~\ref{lem:description}(c) for $\tau=\ell$ and $\tau'=r$ to conclude that $\bigmat{Q}_\ell$ and $B$ are equal in position $vw$. Part (d) of the same lemma, for $\tau'=\ell$ and $\tau=t-1$, implies that  $\bigmat{Q}_{t-1}$ and $B$ are equal in position $vw$. To conclude the proof, since $Q_t= Q_{\ell} \oplus_{S,L,R} Q_r $, we know that the entry $vw$ in $B$ is $1$ if $(i,j) \in S$, where $i$ is the type of $v$ and $j$ is the type of $w$, otherwise it is 0, as required.
\end{proof}

Lemma~\ref{lemma_bigmat} implies that, when processing $Q_t$, the matrix $M$ may be constructed by first taking the disjoint union of the matrices transmitted by both children
and then updating the entries $vw$ where $v$ and $w$ are \typeii{} vertices of different sides. Precisely, if $(i,j) \in S$,
$v$ is a {\typeii} vertex in $G(Q_\ell)$ with label $i$,
and $w$ is a {\typeii} vertex in $G(Q_r)$
with label $j$, we place a one in the row (column) of $v$ and column (row) of $w$.
Observe that the unique label condition imposed on $M^{(2)}$
implies that at most one pair of entries will be modified for any element of $S$.
Let $F$ be the block of ones defining these edges. Then node $Q_t$ starts with
\begin{equation}
\label{eq:boxofQ}
\M{Q_t} =
\left(
  \begin{array}{cc}
\begin{array}{|c|c|}\hline
&  \\
M_\ell^{(0)}& \hspace*{3mm} M_\ell^{(1)} \hspace*{3mm}\\[3mm]\hline
&  \\
M_\ell^{(1)T}  & \hspace*{3mm} M_\ell^{(2)} \hspace*{3mm}\\
&  \\[3mm]\hline
\end{array}
 &
 \begin{array}{cc}
&  \\
\multicolumn{2}{c} 0 \\[3mm]
&  \\
0 \hspace*{2mm} &  \hspace*{5mm} F \\
&  \\[3mm]
\end{array}\\
\begin{array}{cc}
&  \\
\multicolumn{2}{c} 0 \\[3mm]
&  \\
0 \hspace*{2mm} &  \hspace*{5mm} F^T \\
&  \\[3mm]
\end{array}
& \begin{array}{|c|c|}\hline
&  \\
M_r^{(0)}& \hspace*{3mm} M_r^{(1)} \hspace*{3mm}\\[3mm]\hline
&  \\
M_r^{(1)T}  & \hspace*{3mm} M_r^{(2)} \hspace*{3mm}\\
&  \\[3mm]\hline
\end{array} \\
  \end{array}
\right)
\end{equation}
(Here, and in the remainder of the description of the procedure CombineBoxes, we abuse the notation slightly and use $\M{Q_t}$ to refer to the matrix $M$ obtained by merging the boxes transmitted by the children, even before it is a proper matrix for the box $b_{Q_t}$.)
Next $Q_t$ relabels the rows of
$\M{Q_\ell}$ and $\M{Q_r}$,
using the functions $L$ and $R$,
respectively. Note that the {\typei} rows from
$\M{Q_\ell}$ and $\M{Q_r}$
are still {\typei} in (\ref{eq:boxofQ}).
Using permutations of rows and columns we combine the
$k'_{\ell}$ {\typei} rows from $\M{Q_\ell}$
with the $k'_{r}$ {\typei} rows of $\M{Q_r}$,
and combine the $k''_{\ell}$ {\typeii} rows in $\M{Q_\ell }$
with the $k''_{r}$ {\typeii} rows in $\M{Q_r}$,
obtaining a matrix $\M{Q}$ in the form of (\ref{eq:formofM}).
In this matrix, $k' = k'_{\ell} + k'_r$
and $k'' = k''_{\ell} + k''_r$.
By induction, we can assume that $k' \leq k''$, however we are not guaranteed that $k'' \leq k$.

The next three lemmas explain how to produce a box from $\M{Q_t}$ in a way that Lemma~\ref{lem:description} holds.
This requires achieving
$k'' \leq k$ without losing
$k' \leq k''$
using congruence operations. This may also create new permanent elements in the diagonal submatrix $D$ of $\bigmat{Q}_t$.
The matrix is shrunk in three steps.
First, $k'' \leq k$ is obtained by transforming \typeii{} rows into \typei{} rows.
Next, if necessary, $M^{(0)}$ is made $0$ or empty, creating diagonal elements.
Finally, $k' \leq k''$ is obtained, creating additional diagonal elements.

\begin{lemma}
\label{lem:k''}
If $\M{Q_t}$ contains two {\typeii} rows with the same label, then we can produce a matrix that is congruent to it where $k''$ decreases by one and $k'$ increases by one.
\end{lemma}
\begin{proof}
Suppose two {\typeii} rows $j$ and $j'$ have the same label.
Then by Remark~\ref{obs:equalrows} and Lemma~\ref{lem:description} their rows
and columns must agree outside of $\M{Q_t}$.
Therefore in (\ref{eq:bigmatrixBQ}), $\beta^i_j = \beta^i_{j'}$ for all $i$.
Performing the operations
\begin{eqnarray}
R_{j'} & \leftarrow & R_{j'} - R_j \label{lem:subtractrow}\\
C_{j'} & \leftarrow & C_{j'} - C_j \label{lem:subtractcol}
\end{eqnarray}
transforms the {\typeii} row $j'$ into a {\typei} row,
decreasing $k''$ by one, and increasing $k'$ by one.
\end{proof}
We keep applying the operations (\ref{lem:subtractrow}) and
(\ref{lem:subtractcol}) to {\em any pair of {\typeii} } rows with the same label.
By the pigeon-hole principle, this will force $k'' \leq k$.
However, at the end of this step we may have $k' > k''$.
To reduce $k'$ we first make $M^{(0)}=0$.

\begin{lemma}
\label{lem:Mzero}
Given a matrix $\M{Q_t}$ as in~\eqref{eq:formofM}, we can make $M^{(0)}$ zero or empty.
\end{lemma}
\begin{proof}
Let $R_i$ be the $i$-th row and $C_i$ be the $i$-th column of $\M{Q_t}$.
If some diagonal element $m_{ii}$ of $M^{(0)}$ is nonzero, then
subtract $m_{ij} / m_{ii}$ times the
$i$-th row of $\M{Q_t}$ from the $j$-th row.
Do likewise for the columns.
In other words, operate on $\M{Q_t}$ as follows for all $j \neq i$.
\begin{eqnarray*}
R_j & \leftarrow & R_j - \frac{m_{ij}}{m_{ii}}  R_i, \\
C_j & \leftarrow & C_j - \frac{m_{ij}}{m_{ii}}  C_i.\\
\end{eqnarray*}
Note that due to the zero extension of the $j$-th row (see Lemma~\ref{lem:description}(b)), when the same operations are performed to produce  $\widetilde{B_Q^+}_t$, all the work is restricted to the small matrix $\M{Q_t}$, and
the {\em shape} of $\widetilde{B_Q^+}_t$ does not change.
We have removed a row and a column from $M^{(0)}$, and
$\widetilde{B_Q^+}_t$ has one more diagonal element $m_{ii}$.

If the diagonal of $M^{(0)}$ is 0, but an off-diagonal element $m_{ij}$ is nonzero, then
do the operations
 \begin{eqnarray*}
R_j & \leftarrow & R_j + \frac{1}{2}  R_i,  \\
C_j & \leftarrow & C_j + \frac{1}{2}  C_i,
\end{eqnarray*}
followed by
 \begin{eqnarray*}
R_i & \leftarrow & R_i - R_j,  \\
C_i & \leftarrow & C_i - C_j .
\end{eqnarray*}
The relevant entries of $M^{(0)}$ are modified as follows:
\begin{equation}
\label{eq:twobytwotrick}
\begin{pmatrix}
0 & m_{ij} \\
m_{ij} & 0
\end{pmatrix}
\rightarrow
\begin{pmatrix}
0 & m_{ij} \\
m_{ij} & m_{ij}
\end{pmatrix}
\rightarrow
\begin{pmatrix}
-m_{ij} & 0\\
0 & m_{ij}
\end{pmatrix}.
\end{equation}
Now there are two nonzero diagonal elements whose rows can be
annihilated as above, producing two more diagonalized rows.
Again, note that due to the surrounding zero pattern,
only $\M{Q_t}$ is modified when these operations are performed,
so that the portion of $\widetilde{B_Q^+}_t$ outside of $\M{Q_t}$ remains unchanged.
\end{proof}
The operations in Lemma~\ref{lem:Mzero}
create diagonal elements, decrease $k'$ and make $M^{(0)}$ either $0$ or empty.
If we now have $k' \leq k''$ we achieve our goal.
Otherwise the next lemma
ensures this, creating $k' - k''$ zero diagonal elements.
\begin{lemma}
\label{lem:k'}
Consider the matrix $\M{Q_t}$ obtained after applying the operations of Lemma~\ref{lem:Mzero}. Then we can make $k' \leq k''$.
\end{lemma}
\begin{proof}
This is obvious if $M^{(0)}$ is empty, so assume that $M^{(0)}=0$ and $k' > k''$.
With simple row operations (subtracting multiples of one row from another one,
and doing permutations of rows) $M^{(1)}$ can be made upper triangular.
Doing the same operations on the columns of $M^{(1)T}$, the matrix $\widetilde{B_Q^+}_t$ remains symmetric.
Since $k' > k''$
at least $k'-k''$ rows of $M^{(1)}$ have become 0.
Thus the diagonal part of $\widetilde{B_Q^+}_t$ has grown by $k'-k''$ with zeros in the diagonal.
The new $k'$ is at most $k''$ meaning that $M^{(1)}$ is at most as high as wide.
\end{proof}

After applying the computations in Lemma~\ref{lem:k''}, Lemma~\ref{lem:Mzero}
and Lemma~\ref{lem:k'} if necessary,  node $Q_t$ is guaranteed
that $k' \leq k'' \leq k$, and returns $\M{Q_t}$. This concludes the description of the procedure CombineBoxes, which appears in Figure~\ref{fig:algo-code}. Note that all operations performed in the lemmas are congruence operations.

\begin{figure}[t]
{\tt
\begin{tabbing}
aaa\=aaa\=aaa\=aaa\=aaa\=aaa\=aaa\=aaa\= \kill
     \> \texttt{CombineBoxes($b_{Q_\ell}$,$b_{Q_r}$)}\\
     \> {\bf input}:  two $k$-boxes $b_{Q_\ell}$ and $b_{Q_r}$ \\
     % \> {\bf returns}: $2 \times 2$ or empty box $M$ \\
     \> {\bf output}: a $k$-box $b_Q$ \\
     \>  form matrix $M$ in (\ref{eq:boxofQ});  \\
     \>  relabel rows with functions $L$ and $R$; \\
     \>  combine {\typei} rows (columns), combine {\typeii} rows (columns);\\
     \>  ensure {\typeii} rows have distinct labels (Lemma~\ref{lem:k''}); \\
     \>  {\bf if} $k'>k''$ {\bf then} make $M^{(0)}$ empty or zero (Lemma~\ref{lem:Mzero}) and {\bf output} diagonals;\\
     \>  {\bf if} $k'>k''$ {\bf then} ensure $k' \leq k''$ (Lemma~\ref{lem:k'}) and {\bf output} diagonals;\\
\end{tabbing}
}
\caption{\label{fig:algo-code} Procedure \texttt{CombineBoxes}.}
\end{figure}

%The computation can be organized as a post-order traversal of the parse tree of $Q_G$.

To conclude the description of the algorithm, we can assume that at the root
$$
Q_G = Q_\ell \oplus_{L,R,S} Q_r,
$$
$L$ and $R$ map all vertices to the same label,
as labels are no longer needed.
After applying the operation in Lemma~\ref{lem:k''}, we will obtain $k'' = 1$.
Applying the operations in Lemma~\ref{lem:Mzero}
and Lemma~\ref{lem:k'} will make $M^{(0)}$ either zero or empty.
If it is empty then the $1 \times 1$ matrix $M^{(2)}$
contains the final diagonal element. Otherwise $\M{Q}$ is a $2 \times 2$ matrix having form
$$
\begin{pmatrix}
0 & a \\
a & b
\end{pmatrix}
$$
and can be made fully diagonal using the transformations in
Lemma~\ref{lem:Mzero}. This is what we call \texttt{DiagonalizeBox} in Figure~\ref{fig:highlevel}.

\begin{remark}
When defining \texttt{diagonalize}, we added several features to simplify the description, but which are not crucial for the algorithm to work, and which would not necessarily be used in an efficient implementation of the algorithm.
\begin{itemize}
\item[(a)] Since the matrix $\bigmat{Q}$ is not computed, one may easily write \texttt{diagonalize} as a recursive algorithm.
\item[(b)] It is not necessary to perform permutations of rows and columns to separate them according to type, it suffices to keep track of the vertices of each type in matrices $\M{Q_t}$.
\item[(c)] The requirement that all vertices are relabelled with the same label at the root node is not crucial. The root could just produce an arbitrary box from the boxes transmitted by its children, and the final step of the algorithm, DiagonalizeBox, could just diagonalize this box with congruence operations in any way.
\item[(d)] When performing \texttt{CombineBoxes}, and after applying Lemma~\ref{lem:k''}, the algorithm uses Lemma~\ref{lem:Mzero} to make $M^{(0)}$ zero or empty if $k'>k''$. In fact, it is not necessary to get to this point, the procedure could have stopped at any point after he operations of Lemma~\ref{lem:Mzero} have produced enough diagonal elements so that $k'\leq k''$. Morever, if we still have $k'>k''$ after making $M^{(0)}$ zero or empty, \texttt{CombineBoxes} asks us to ensure that $k'\leq k''$ using operations of Lemma~\ref{lem:k'}. We could have asked the algorithm to perform more operations to turn $M^{(1)}$ into a matrix with more structure (upper triangular, for example).
\end{itemize}
\end{remark}

To prove the correctness of the algorithm, we shall now see that Lemma~\ref{lem:description} holds by induction on $t$.
\begin{proof}[Proof of Lemma~\ref{lem:description}] Assertion (e) is trivial, as LeafBox only produces rows of {\typeii} and no operation in CombineBoxes may turn a row of {\typei} into {\typeii}.

For assertion (a), the fact that diagonal elements are obtained through some application of CombineBoxes is obvious. Now, let $v$ be a row that was diagonalized when processing $Q_{\tau'}$ for some $\tau'\leq t-1$. First observe that any entry $vw$ with $w \neq v$ is 0 in $\bigmat{Q}_{t-1}$. This is trivial for $\tau'=t-1$ and holds by induction for $\tau'\leq t-2$. Such an entry remains 0 in $\bigmat{Q}_{t}$ because all the operations in Lemmas~\ref{lem:k''},~\ref{lem:Mzero} and~\ref{lem:k'} involve undiagonalized rows and columns in $G(Q_t)$. Then, when a multiple of some column $w'$ is added to $w$, it adds some multiple of $vw'$ to $vw$, and both are 0.

Regarding assertion (b), suppose that row $v$ has \typei{} in $\M{Q_{\tau'}}$, for some $\tau' \in \{1,\ldots,t\}$. Suppose also that $w$ is a row such that, for all $j \in \{\tau',\ldots,t\}$,  $v$ and $w$ are not simultaneously in $G(Q_j)$. We wish to show that entry $vw$ in $\bigmat{Q}_t$ is $0$. By part (a), this is immediate if $v$ or $w$ has been diagonalized up this stage, so assume that this is not the case. First suppose that $\tau'=t$, so that $w \notin G(Q_t)$. The first possibility is that the algorithm turned $v$ to {\typei} at this stage, performing an operation of Lemma~\ref{lem:k''}, which zeroed entry $vw$ in $\bigmat{Q}_t$. After this, {\typei} rows are only added to one another (see Lemmas~\ref{lem:Mzero} and~\ref{lem:k'}), so that such zero entries are not altered. The second option is that $v$ was already of {\typei} in an earlier stage $t'$. Since $w$ has not been diagonalized, the same assertion (b) must hold replacing $\tau'=t$ by $\tau'=t'$, and we would be in the case $\tau'<t$.

For $\tau'<t$, by induction the entry $vw$ in $\bigmat{Q}_{t-1}$ is $0$. If $v \in G(Q_t)$, then the fact that $v$ has not been diagonalized and assertion (e) ensure that $v$ has \typei{} in $\M{Q_t}$. Note that $w \notin G(Q_t)$. Adding the multiple of any row $v'$ to $v$ does not change entry $vw$, as $v'$ must have {\typei} and therefore entry $v'w$ is zero by Lemma~\ref{lem:k''} (if $v'$ turned to {\typei} at stage $t$), or by induction (if this happened at an earlier stage). Next assume that $v \notin G(Q_t)$. Then the only way to change entry $vw$ is to add a multiple of some column $w' \in G(Q_t)$ to column $w$, and in particular $w',w \in G(Q_t)$. However, we may use the induction hypothesis for $v$ and $w'$ to conclude that the entry $vw'$ in $\bigmat{Q}_{t-1}$ is also $0$, and hence no change occurs.

Assertion (d) holds for the following reason. Assume that rows $v$ and $w$ are not in $\M{Q_j}$ for all $\tau'<j \leq t$, where $0 \leq \tau'< t$.
By induction, $\bigmat{Q}_{\tau'}$ and $\bigmat{Q}_{t-1}$ are equal in position $vw$ (recall that $\bigmat{Q}_{0}=B$). Note that this position could be changed only if the multiple of some row was added to row $v$ or the multiple of some column was added to column $w$, but this cannot happen, as all the operations in Lemmas~\ref{lem:k''},~\ref{lem:Mzero} and~\ref{lem:k'} involve rows and columns in $G(Q_t)$, and neither $v$ nor $w$ is in $\M{Q_t}$.

To conclude, we prove assertion (c). Suppose that row $v$ is of \typeii{} in $\M{Q_t}$, $w \notin G(Q_t)$, and $w$ has \typeii{} in $\submat{Q}_{j}$ whenever $w \in G(Q_{j})$ for some $j < t$. First observe that $v$ and $w$ cannot be both in some $G(Q_j)$ for $j<t$: by the ordering of the nodes of the parse tree, since $v \in \M{Q_t}$, when $j <t$, $v$ belongs to precisely those $G(Q_{j})$ on the path from its leaf to $Q_t$.  Since $w \notin G(Q_t)$, it cannot belong to any of these $G(Q_{j})$.

We claim that $B$ and $\bigmat{Q}_{t-1}$ are equal in position $vw$. Let $\tau'$ be the largest $j<t$ such that either $v$ or $w$ is in $\M{Q_j}$ ($\tau'=0$ if there is no such $j$). If $\tau'=0$, the claim holds by substituting $\tau'=0$ and $\tau = t-1$ in part (d). If $\tau'>0$, then either $v$ or $w$ is of {\typeii} in $\M{Q_{\tau'}}$ ($w$ by hypothesis, $v$ by part (e)). By induction, we see that $B$ and $\bigmat{Q}_{\tau'}$ are equal in position $vw$. By the choice of $\tau'$ we may use assertion (d) for $\tau'$ and $\tau=t-1$, and we see that $\bigmat{Q}_{\tau'}$ and  $\bigmat{Q}_{t-1}$ are equal in position $vw$, which proves that $B$ and $\bigmat{Q}_{t-1}$ and agree at position $vw$. To complete the proof, we show that $\bigmat{Q}_{t-1}$ and $\bigmat{Q}_{t}$ are equal in position $vw$. Note that no row (or multiple thereof) can be added to row $w$, as $w \notin G(Q_t)$. Moreover, if a multiple of a row $u$ is added to row $v$, then, because $v$ has \typeii{} in $\M{Q_t}$ and by the description of \texttt{CombineBoxes}, $u$ must have been of \typei{} at the moment when the operation has been performed. There are two possibilities: either $u$ had \typei{} in $\M{Q_{\tau'}}$, for some $\tau'<t$, (and thus $\bigmat{Q}_{t-1}$ is 0 in position $uw$ because of part (b) with $\tau=t-1$) or $u$ became of \typei{} at stage $t$ (this happened due to an operation of Lemmas~\ref{lem:k''}, and performing the corresponding operation on the matrix derived from $\bigmat{Q}_{t-1}$ turns the entry $uw$ to 0. In both cases, adding a multiple of $u$ to $v$ does not change the entry $vw$ in $\bigmat{Q}_{t}$, as required.
\end{proof}

\begin{theorem}
Let $G$ be a graph with adjacency matrix $A$, given by a \skex\ $Q_G$ with parse tree $T$, and let $c \in \mathbb{R}$. Algorithm \texttt{diagonalize} correctly outputs the diagonal elements of a diagonal matrix congruent to $B= A - cI$. Moreover, this is done in $O(k^3 n)$ operations.
\end{theorem}

\begin{proof}
It is clear that the operations performed by \texttt{diagonalize} are congruence operations. Indeed, besides simultaneous permutations of rows and columns,
the operations performed (in Lemma~\ref{lem:k''}, Lemma~\ref{lem:Mzero} and Lemma~\ref{lem:k'})
only add multiples of some row $i$ to some other row $j$, followed by adding the same multiple of
column $i$ to column $j$. We now argue that the elements that the algorithm outputs are precisely the diagonal elements of a diagonal matrix congruent to $B= A - cI$. As above, let $\bigmat{Q}_t$ be the $n \times n$ matrix obtained from $B$ by performing the same congruence operations on $B$ that are actually performed on the corresponding rows of matrices $\M{Q_t}$ up to stage $t$. Lemma~\ref{lemma_bigmat} ensures that, at the beginning of each stage $t$, the algorithm always starts from a submatrix $M$ of $\bigmat{Q}_{t-1}$. Also, Lemma~\ref{lem:description} guarantees that equation~\eqref{eq:bigmatrixBQ} holds at the end of each stage: part (a) ensures that a row that was diagonalized at some stage cannot be modified at later stages. If $v \in G(Q_t)$ and $w \notin G(Q_t)$ (and has not been diagonalized), part (b) ensures (applying it twice with the roles of $v$ and $w$ interchanged) that the element $\beta_{i}^j$ corresponding to the entry $vw$ in $\bigmat{Q}_t$ is equal to 0 if $v$ has \typei{} or if $w$ has \typei{} in $\M{Q_{\tau}}$ for some $\tau<t$. Finally, part (c) ensures that in the remaining cases the entry $vw$ is equal to the corresponding entry in $B$. We conclude that, at each stage, the elements in the output of \texttt{CombineBoxes} (which are called \emph{diagonals} in Figure~\ref{fig:algo-code}) are indeed diagonal elements of the final matrix produced. The only nonzero elements of the final matrix that are not of this form are obtained by the application of \texttt{DiagonalizeBox}, which clearly outputs the final diagonal elements of a diagonal matrix $D$ congruent to $B$.

A time bound of $O(k^3 n)$ is straightforward, because the parse tree has $O(n)$
nodes ($2n-1$ to be precise), $n = |V|$ leaves, and $n-1$ internal nodes with 2 children each.
At each node, since the algorithm acts on an $O(k) \times O(k)$ matrix, the number of row and column operations performed by Lemma~\ref{lem:k''}, Lemma~\ref{lem:Mzero} and Lemma~\ref{lem:k'} is $O(k^2)$. Each such operation requires $O(k)$ sums and products.
%But there is an expensive Gaussian elimination process happening in the procedure handled by Lemma~\ref{lem:k'}. Here a matrix of size up to $2k$ is put into upper triangular form. This could be done theoretically in time $k^{\omega +o(1)}$, where $\omega$ is the exponent of matrix multiplication, but this is not necessary.
%The bound of $O(k^2 n)$ is obtained by a better accounting for the time and making sure that no time is wasted in the implementation of Lemma~\ref{lem:k'}. The matrix $M^{(1)}$ can be maintained as an upper triangular matrix. Inserting one row vector incurs a cost of $O(k^2)$. It does not matter that up to $k$ row vectors are inserted in one node, because every row vector is inserted only once, and there are only $n$ rows, one for each vertex. The Gaussian operations in Lemma~\ref{lem:k''} and Lemma~\ref{lem:Mzero} can also be done in $O(k^2 n)$.
\end{proof}
%Having avoided fast matrix multiplication,

Now, we can apply Sylvester's law of inertia. Symmetric matrices over the reals have $n$ real eigenvalues
(with multiplicities). The inertia of a symmetric real matrix $B$ is the triple $(n_+, n_0,n_-)$ giving the
number of eigenvalues of $B$ that are positive, zero, and negative respectively.
Sylvester's law says that congruent real symmetric matrices have the same inertia.

\begin{corollary}
The number of eigenvalues of $A$ in a real interval $I$ can be computed in time $O(k^3 n)$ for graphs of \cw\ $k$.
\end{corollary}

\begin{proof}
The eigenvalues of $B = A-c l$ are obtained by subtracting $c$ from the eigenvalues of $A$.
To compute the number of eigenvalues of $A$ in a nonempty interval $I=[a,b]$, for instance, we run our algorithm with $c=a$ and $c=b$.
From the output, we see the numbers of positive, zero, and negative diagonal elements.
Let them be $(n_+^a, n_0^a, n_-^a)$ and $(n_+^b, n_0^b, n_-^b)$.
Then, obviously the number of eigenvalues in $(a,b)$ is $n_-^b - n_-^a-n_0^a$. Of course, there is nothing special about open (or bounded) intervals, and a similar argument would lead to the number of eigenvalues in any given real interval.
\end{proof}

It is worth stating that the whole algorithm is very fast, as there are no large constants hidden in the $O$-notation. In fact, it is not hard to modify Algorithm \texttt{diagonalize} to obtain an algorithm that requires only $O(k^2n)$ operations. Instead of performing the operations of Lemmas~\ref{lem:k''},~\ref{lem:Mzero} and~\ref{lem:k'} successively, we may keep $M^{(0)}$ equal to 0 or empty and $M^{(1)}$ as an upper triangular matrix, namely a matrix $(a_{ij})$ such that $a_{ij}=0$ for every $i>j$. We insert row vectors (namely a {\typeii} vertex that becomes {\typei}) one at a time. At the end of the step, inserting such a vector may incur a cost of $O(k^2)$ in terms of $M^{(1)}$ operations. It does not matter that up to $k$ row vectors are inserted in one node, because every row vector is inserted only once, and there are only $n$ rows, one for each vertex in the graph.

Every time a new row $v$ is inserted (and to simplify our description we assume that it becomes the first row of the new matrix), instead of performing the operations of Lemma~\ref{lem:Mzero} using the diagonal element, we choose the largest index $j$ such that $j$ has {\typei} and position $vj$ is nonzero. We use this element to eliminate all the other nonzero elements of row (and column) $v$ in $M^{(0)}$. The choice of $j$ ensures that the operations performed do not destroy the upper triangular nature of $M^{(1)}$ (if we ignore row $v$). Once this is done, we replace the two remaining nonzero off-diagonal elements by diagonal elements as in Lemma~\ref{lem:Mzero} and either diagonalize new rows or get $M^{(0)}=0$. Then we perform operations of Lemma~\ref{lem:k'} (and possibly exchange rows and columns) to turn $M^{(1)}$ into the upper triangular matrix mentioned in the previous paragraph. The cases where the diagonal element is the single such nonzero element in this row, and where all the elements of this row are zero may be treated similarly.

\section{Example}
\label{sec:example}

To see how the algorithm acts on a concrete example, we go back to the graph of Figure~\ref{fig:int1}, whose parse tree is given in Figure~\ref{fig:parsetree}. We apply the algorithm to the graph defined by this slick $2$-expression for $c=0$.
Since $k=2$, the boxes created by the leaves may be of the following two types:
$$1(v): [k^{\prime},k^{\prime \prime},M,labels]=\left[ 0,1,(0),(1)\right],~\textrm{ or } 2(v): \left[ 0,1,(0),(2)\right],$$
This means that $k^{\prime}= 0$  (thus $ M^{(0)}$ is empty) and that $k^{\prime\prime}=1$ and $M^{(2)}=(0)$.

The nodes $B$, $E$ and $F$ of the parse tree perform identical operations to produce
$$\left[0,2, \left( \begin {array}{rr} 0&1\\ 1&0
\end {array} \right) , \left( \begin {array}{r} 1\\ 2\end {array} \right) \right],$$
meaning that $ M^{(0)}$ is empty, $k^{\prime\prime}=2$, and
 $M^{(2)}= \left( \begin {array}{rr} 0&1\\ 1&0
\end {array} \right)$ and the labels of its rows are 1 and 2, respectively. Up to this point, we did not need to apply Lemmas~\ref{lem:k''},~\ref{lem:Mzero} or~\ref{lem:k'}.

When the algorithm starts processing node $D$, it uses the boxes produced by $E$ and $F$ together with $S=\{(2,2)\}$ to form the matrix
$$
\left( \begin {array}{rrrr}  0&1&0&0\\ 1&0&0&1
\\ 0&0&0&1\\ 0&1&1&0\end {array}
 \right).
$$
It then uses row and column operations described in Lemma~\ref{lem:k''} (with $(j',j)$ given by $(1,3)$ and $(2,4)$) to produce
$$
\left[2,2, \left( \begin {array}{rr;{.5pt/1pt}rr}
0&2&0&-1\\2&-2&-1&1
\\ \hdashline[.5pt/1pt] 0&-1&0&1\\  -1&1&1&0\end {array}
  \right) , \left( \begin {array}{c} 1\\ 2
\\ 1\\ 2\end {array} \right)
\right].
$$
This means that
$$k^{\prime}=k^{\prime \prime}= 2, ~M^{(0)}= \left(
\begin {array}{rr} 0&2\\ 2&-2 \end{array}\right) ,~\textrm{ and }M^{(2)}= \left(
\begin {array}{rr} 0&1\\ 1&0 \end{array}\right).$$
However, there is no need to resort to Lemmas~\ref{lem:Mzero} or~\ref{lem:k'}, as the relation
$k^{\prime}\leq k^{\prime\prime}\leq k \leq 2$ holds.

Next, node $C$ receives boxes from a leaf and node $D$ that, together with $S=\{(1,2)\}$, produce the following matrix and vector of labels (recall that the entries created by $S$ only affect \typeii{} rows and columns.):
$$ \left( \begin {array}{rrrrr} 0&0&0&0&1\\ 0&0&2&0&-1
\\ 0&2&-2&-1&1\\ 0&0&-1&0&1
\\ 1&-1&1&1&0\end {array} \right),~~\left( \begin {array}{r}1\\ 1\\ 2\\ \ 1\\ 2 \end{array} \right).$$
Vertices of the right component are relabelled $1\rightarrow 2$ and, as in the description of the algorithm, we exchange rows (and columns) 1 and 3 to keep \typei{} and \typeii{} rows together. This leads to
$$M=\left( \begin {array}{rr;{.5pt/1pt}rrr} -2&2&0&-1&1\\ 2&0&0&0&
 -1\\ \hdashline[.5pt/1pt]  0&0&0&0&1\\ -1&0&0&0&1
\\ 1&-1&1&1&0\end {array} \right),
\textrm{ and labels }\left( \begin {array}{r}2\\  2\\  1\\  2\\ 2 \end{array} \right).$$

At this point $M^{(2)}$ is a $3 \times 3$ matrix and $k^{\prime\prime}=3 > k=2$.
We first apply Lemma~\ref{lem:k''} to $j'=4$ and $j=5$ in order to preserve the uniqueness of labels for \typeii{}. Moreover, as in the above description of the algorithm, we exchange rows $3$ and $4$ to keep types together:
$$M=\left( \begin {array}{rrr;{.5pt/1pt}rr} -2&2&-2&0&1\\ 2&0&1&0&
-1\\ -2&1&-2&-1&1\\ \hdashline[.5pt/1pt] \noalign{\medskip}0&0&-1&0&1
\\ 1&-1&1&1&0\end {array} \right)
\textrm{ with labels }
\left( \begin {array}{r}2\\  2\\  1\\  2\\  2 \end{array} \right).$$
Now $k^{\prime\prime}=2<3=k^{\prime}$, so we need to reduce
$$M^{(0)}= \left( \begin {array}{rrr} -2&2&-2 \\ 2&0&1 \\
-2&1&-2\end{array} \right).$$
The diagonal element of the first line of $M^{(0)}$ is -2, hence we can diagonalize that row as described in Lemma~\ref{lem:Mzero}, so that $M$ becomes
$$\left( \begin {array}{r;{2pt/2pt}rr;{.5pt/1pt}rr} -2&0&0&0&0\\  \hdashline[2pt/2pt]   0&2&-1&0&0
\\  0&-1&0&-1&0\\ \cdashline{2-5}[.5pt/1pt]   0&0&-1&0&1
\\   0&0&0&1&\frac{1}{2}\end {array} \right)
$$
We have found the diagonal element $-2$, which is stored by the algorithm, and the first row (and column) are removed from $M$. We may repeat the argument for the next two rows of $M^{(0)}$, which leads to diagonal elements $2$ and $-\frac{1}{2}$. At the end of the process, we have $k^\prime=0$, $k^{\prime\prime}=2$, and $M^{(2)}=\left( \begin{array}{rr}
                     2 & 1 \\
                     1 & \frac{1}{2}
                   \end{array}\right)$. Node $C$ transmits the box
$$\left[0,2,\left( \begin{array}{rr}
                     2 & 1 \\
                     1 & \frac{1}{2}
                   \end{array}\right),\left(\begin{array}{r}
                                        1 \\
                                        2
                                      \end{array}\right)\right].$$

We finally process node $A$, which combines the boxes produced by $B$ and $C$. Note that, in this matrix, we have $k'=0$, $k''=4$ and, after adding the edge of $S=\{(1,2)\}$ and then relabeling, we have
$$\left( \begin {array}{rrrr} 0&1&0&1\\ 1&0&0&0
\\ 0&0&2&1\\ 1&0&1&\frac{1}{2}\end {array}
 \right) \textrm{ with labels }
\left( \begin {array}{r}1\\  1\\  1\\  1 \end{array} \right).$$
Applying Lemma~\ref{lem:k''} for $(j',j) \in \{(1,4),(2,4),(3,4)\}$, we obtain the following matrix, with $k'=3$ and $k''=1$:
$$\frac{1}{2} \cdot \left( \begin {array}{rrr;{.5pt/1pt}r} -3&1&-3&1\\ 1&1&-1&-1
\\ -3&-1&1&1\\ \hdashline[.5pt/1pt] 1&-1&1&1\end {array}
 \right).
$$

The algorithm now performs operations of Lemma~\ref{lem:Mzero}, which leads to
$$\left( \begin {array}{cccc} -\frac{3}{2}&0&0&0\\ 0&\frac{2}{3}&0&0
\\ 0&0&\frac{1}{2}&0\\ 0&0&0&0\end {array}
 \right),
$$
and the final diagonal values are $D=(-2,2,-\frac{1}{2}, -\frac{3}{2},\frac{2}{3},\frac{1}{2},0)$. This means that the graph has three negative eigenvalues, three positive eigenvalues and 0 is an eigenvalue with multiplicity one. In fact, the actual spectrum may be approximated by
$$ -1.9098;-1.6180;-1.2726;0;0.6180;0.8692;3.3132.$$

\section{Applications}
\label{sec:appl}
Section~\ref{sec:algorithm} showed how to compute, in linear time,
the number of eigenvalues of a graph in a given interval.
Here we give some alternate ways to bound this number.
If $\lambda$ is an eigenvalue in a graph $G$, we denote its multiplicity with $m_G (\lambda)$.
If $G$ is represented by a {\skex} $Q$
we will also use $m_Q (\lambda)$ to denote the multiplicity of $\lambda$ in $G$.
We can deduce some simple properties about multiplicity
by studying the behavior of algorithm {\scwalgo}.

\begin{remark}
\label{rem:multineqal}
If $G$ has {\skex} $Q = Q_\ell \oplus_{L,R,S} Q_r$
then for any eigenvalue $\lambda$
$$
m_G (\lambda) \leq m_{Q_\ell}(\lambda) + m_{Q_r}(\lambda) + 4k
$$
\end{remark}
\begin{proof}
Note that $m_G (\lambda)$ is exactly the number of zero diagonal values produced by
executing {\scwalgo}$(Q, \lambda)$.
Assume the call to the left subtree $Q_{\ell}$
produces $m_\ell$ zeros,
the call to the right subtree $Q_r$
produces $m_r$ zeros,
and $m$ zeros are produced
from the matrix $M$ in (\ref{eq:boxofQ}).
Then we have
$m_G (\lambda) = m_\ell + m_r + m$.
Since  $m_\ell \leq m_{Q_\ell}(\lambda)$ and $m_r \leq m_{Q_r}(\lambda)$ we have
$$
m_G (\lambda) \leq m_{Q_\ell}(\lambda) + m_{Q_r}(\lambda) + m
$$
The result follows since $m$ is precisely the nullity of $M$ whose size
is at most $4k \times 4k$.
\end{proof}

It follows that if
$m_G (\lambda) > 4k$
then either $m_{Q_\ell}(\lambda) >0$ or $m_{Q_r}(\lambda) >0$.
Another immediate consequence of Remark~\ref{rem:multineqal} is the following.
\begin{remark}
\label{rem:multint}
If $I$ is an interval in which neither $Q_\ell$ or $Q_r$ contain eigenvalues
and $G$ has {\skex} $Q = Q_\ell \oplus_{L,R,S} Q_r$,
then for each $\lambda \in I$ we have $m_G (\lambda) \leq 4k$.
\end{remark}

It is interesting to relate
$m_G(\lambda)$ to the operations in {\scwalgo}$(Q, \lambda)$.
During execution, only the operations in
Lemmas~\ref{lem:Mzero}~and~\ref{lem:k'}
can create diagonal elements.
A careful look at Lemma~\ref{lem:Mzero} reveals that
if $M^{(0)} \not= 0$
and has rank $r$,
then it creates $r$ {\em nonzero} diagonal elements.
When $M^{(0)} = 0$ and $k' > k''$,
the operations in Lemma~\ref{lem:k'}
produce $k' - k''$ zero diagonal elements.
If $z$ denote the number of zeros created in
Lemma~\ref{lem:k'} (in the entire algorithm),
since the algorithm may return a  $2 \times 2$
matrix at the end of its execution, we must have
$z \leq m_G(\lambda) \leq z + 2$.

We can also bound the {\em total} number of eigenvalues in an interval.

\begin{theorem}
\label{thr:generalint}
If $G$ has {\skex} $Q = Q_\ell \oplus_{L,R,S} Q_r$, where $Q_\ell$ and $Q_r$ generate
graphs having no eigenvalues in $(a,b)$, then $G$ has at most $8k$ eigenvalues in $(a,b)$.
\end{theorem}
\begin{proof}
Let
$(n_+^{a}, n_0^{a}, n_-^{a})$
be the number of positive, zero and negative diagonal elements produced by
{\scwalgo}$(Q, a)$, and let
$(n_+^{b}, n_0^{b}, n_-^{b})$
be the number of positive, zero and negative diagonal elements produced by
{\scwalgo}$(Q, b)$.
Then, using Sylvester's Law of Inertia, the number of eigenvalues of $G$ in
$(a,b)$ is
$$
n_+^{a} - n_{ge}^b
$$
where
$n_{ge}^b = n_0^{b} + n_+^{b}$.

During the execution of
{\scwalgo}$(Q, a)$, let
${\ell}_+^{a}$ be the
number of positive diagonal values produced by the call
{\scwalgo}$(Q_\ell , a)$,
and let
${r}_+^{a}$ be the number of positive diagonal values
produced by the call
{\scwalgo}$(Q_r , a)$.
Then
\begin{equation}
\label{eq2Lem8k}
n_+^{a} = {\ell}_+^{a} + {r}_+^{a} + n_1
\end{equation}
where $n_1 \leq 4k$,
as the remaining values will come from a
matrix of order at most $4k \times 4k$.

Similarly during
{\scwalgo}$(Q, b)$,
let
$\ell_{ge}^b$
denote the number of nonnegative diagonal values
produced during the left subtree call and let
$r_{ge}^b$ denote
the number of nonnegative diagonal values
produced during the right subtree call.
Then
\begin{equation}
\label{eq3Lem8k}
n_{ge}^{b} = \ell_{ge}^b + r_{ge}^b + n_2
\end{equation}
where $n_2 \leq 4k$.

In the graph defined by $Q_\ell$,
let $n_{\ell,+}^a$ denote the number of eigenvalues greater than $a$.
Since the
$\ell_{+}^{a}$
positive diagonal elements produced
in the left subtree call are valid diagonal elements
in the diagonalization of $Q_\ell$,
and there can be at most $2k$ more,
\begin{equation}
\label{leftpos}
\ell_{+}^{a} \leq n_{\ell,+}^a \leq \ell_{+}^{a} + 2k
\end{equation}
Similarly, letting $n_{\ell,ge}^b$ denote the number of
eigenvalues in $Q_\ell$
greater than or equal to $b$,
we have
\begin{equation}
\label{leftgezero}
\ell_{ge}^b \leq n_{\ell,ge}^b \leq \ell_{ge}^b + 2k
\end{equation}
As $Q_\ell$
has no eigenvalues in $(a,b)$,
we must have $n_{\ell,+}^a = n_{\ell,ge}^b$.
From
(\ref{leftpos}) and (\ref{leftgezero})
we have
\begin{equation}
\label{absdiffleft}
| \ell_{+}^{a} - \ell_{ge}^b| \leq 2k .
\end{equation}
Since $Q_r$ also has no eigenvalues in $(a,b)$,
a similar argument shows
\begin{equation}
\label{absdiffright}
| r_{+}^{a} - r_{ge}^b| \leq 2k .
\end{equation}
Using
(\ref{eq2Lem8k}) and (\ref{eq3Lem8k}) we have
\begin{eqnarray*}
n_+^{a} - n_{ge}^b & = & ({\ell}_+^{a} + {r}_+^{a} + n_1) - (\ell_{ge}^b + r_{ge}^b + n_2) \\
    & = & ({\ell}_+^{a} - \ell_{ge}^b)  + ( {r}_+^{a}  - r_{ge}^b )  + (n_1 - n_2)
\end{eqnarray*}
Using (\ref{absdiffleft}), (\ref{absdiffright})
and the fact that $|n_1 - n_2| \leq 4k$ we conclude that
$n_+^{a} - n_{ge}^b \leq 8k$, completing the proof.
\end{proof}

%%%%%%%%%%%%%%%%%%%%%%%
\begin{figure}[h]
\begin{tikzpicture}
    [scale=1,auto=left,every node/.style={circle,minimum width={.10in}, scale=0.7}]
%%  Nodes on left side are hexagon.  Internal angles are 120 degrees
       \node at (0,.5) [circle,draw, fill=red ] (a1) {} ;
         \node[draw,circle,fill=red   ]  (a2) at ([shift={(120: 1)}]a1) {};
         \node[draw,circle,fill=black ]  (a3) at ([shift={(180: 1)}]a2) {};
         \node[draw,circle,fill=black ]  (a4) at ([shift={(240: 1)}]a3) {};
         \node[draw,circle,fill=black ]  (a5) at ([shift={(-60: 1)}]a4) {};
         \node[draw,circle,fill=red   ]  (a6) at ([shift={(-120: 1)}]a1) {};
%%   Nodes on right side pentagon has internal angles of 108
       \node at (3,0) [circle,draw, fill=red] (b1) {} ;
       \node[draw,circle,fill=red ]  (b2) at ([shift={(90: 1)}]b1) {};
       \node[draw,circle,fill=black ]  (b3) at ([shift={(18: 1)}]b2) {};
       \node[draw,circle,fill=black ]  (b4) at ([shift={(-54:1)}]b3) {};
        \node[draw,circle,fill=black ]  (b5) at ([shift={(-18: 1)}]b1) {};
     \path
%%   Edges on left side:
            (a1) edge [thick](a3)
            (a1) edge [thick](a4)
            (a1) edge [thick](a5)
            (a2) edge [thick](a3)
            (a2) edge [thick](a4)
            (a2) edge [thick](a5)
            (a6) edge [thick](a3)
            (a6) edge [thick](a4)
            (a6) edge [thick](a5)
%%   Edges on right side:
           (b1) edge [thick](b2)
           (b2) edge [thick](b3)
           (b3) edge [thick](b4)
           (b4) edge [thick](b1)
           (b2) edge [thick](b4)
           (b1) edge [thick](b3)
           (b5) edge [thick](b1)
           (b5) edge [thick](b2)
           (b5) edge [thick](b3)
           (b5) edge [thick](b4)
%%   Edges between graphs
            (a1) edge [thick](b1)
            (a1) edge [thick](b2)
            (a2) edge [thick](b1)
            (a2) edge [thick](b2)
            (a6) edge [thick](b1)
            (a6) edge [thick](b2);
\end{tikzpicture}
\caption{\label{fig-cographjoin} Slick join of two cographs.}
\end{figure}
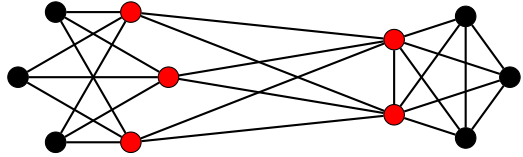
%%%%%%%%%%%%%%%%%%%%%%%

Consider the following construction.
\begin{example}
\label{ex:twocographs}
Let $G_\ell = (V_\ell, E_\ell)$ and $G_r = (V_r E_r)$ be arbitrary cographs,
and let
$W_\ell \subseteq V_\ell$ and
$W_r \subseteq V_r$ be arbitrary sets of vertices in
$G_\ell$ and $G_r$ respectively.
Let $G$ be the graph formed by taking
$G_\ell \cup G_r$ and adding edges
$W = W_\ell \times W_r$.
\end{example}

We claim that the graph $G$
in Construction~\ref{ex:twocographs}
can be defined with a slick $2$-expression
$Q_\ell \oplus_{L,R,S} Q_r$ where $Q_\ell$ and $Q_r$ generate cographs.
Indeed, since any cograph can be constructed with a slick $1$-expression,
there must exist a slick $2$-expression $Q_\ell$ that constructs $G_\ell$, assigning
label $2$ exactly to vertices in $W_\ell$.
Similarly, there exists a slick $2$-expression
$Q_r$ that constructs $G_r$, assigning
label $2$ exactly to vertices in $W_r$.
The graph $G$ can be formed with
$Q_\ell \oplus_{L,R,S} Q_r$ where $S = \{(2,2)\}$.
Since cographs have no eigenvalues in $(-1,0)$ (see \cite{MohammadianT2016}),
by Remark~\ref{rem:multint} and Theorem~\ref{thr:generalint}
it follows that $m_G (\lambda) \leq 4k = 8$
for each $\lambda \in (-1,0)$
and that $G$ has at most $8k = 16$ eigenvalues in the interval $(-1,0)$.
Observe that we may want to place only certain edges in $W$ between
$G_\ell$ and $ G_r$ but this will require additional labels, thus increasing $k$.

Note that Construction~\ref{ex:twocographs}
yields another cograph if $\min\{|W_\ell |, |W_r|\}=0$ or if
$W_\ell = V_\ell$ and $W_r = V_r$ since the disjoint union and the join of two cographs are cographs.
Figure~\ref{fig-cographjoin} illustrates a graph constructed in this way,
and for simplicity the cographs are small.
The left and right cographs are isomorphic to $K_{3,3}$ and $K_5$ respectively.
The vertices in $W_\ell \cup W_r$ are depicted in red.
It is easy to see this graph is not a cograph since it has an induced $P_4$
(for example take the bottom four vertices).

Construction~\ref{ex:twocographs}
seems interesting because the bounds of $8$ and $16$ are independent of both $n$ and $|W|$.
Indeed,
we can place {\em arbitrarily} many edges between
{\em arbitrary} cographs and yet the number of eigenvalues in $(-1,0)$
is bounded by a constant.
Among the graphs $G$ constructed in this way, how close are these bounds?
That is, for $\lambda \in (-1,0)$,
how large can $m_G (\lambda)$ be, and how many
eigenvalues can $G$ have in $(-1,0)$?

To underscore the somewhat unexpected nature of a constant bound when adding arbitrarily many edges, consider taking
an arbitrary graph $H$ having no eigenvalues in $(a,b)$,
and adding $t$ edges to it. Using the interlacing result from~\cite[Thr. 3.9]{Hall2009}, one may show that the number of eigenvalues in $(a,b)$
of the resulting graph $G$ is bounded by $2t$. The following construction shows that this bound is tight.
\begin{example}
\label{ex:2tistight}
Consider the graph $H$ on $4t$ vertices having $2t$ components, each a $K_2$,
whose spectrum is $-1^{2t}; 1^{2t}$.
Form $G$ by adding $t$ edges so that $G$ has $t$ components, each a $P_4$.
Then $G$ has spectrum $-1.618^t; -.618^t; .618^t; 1.618^t$.
In particular, $H$ has no eigenvalues in $(-1,1)$, while $G$ has $2t$.
\end{example}

\section{Concluding Remarks}
\label{sec:conclusion}

Given a graph $G$ and a real number $c$, we have designed an algorithm to find a diagonal matrix congruent to $A-cI$, where $A$ is the adjacency matrix of $G$. This algorithm is very efficient if $G$  has clique-width bounded by a constant $k$. The reason for this is that the algorithm uses an encoding of the graph, called a \emph{slick expression}, which is closely related to the expression defining the graph by means of its clique-width. In fact, for constant $k$, we describe linear-time algorithms to translate one type of expression into the other. We applied this algorithm to the problem of locating graph eigenvalues, which is a basic problem in Spectral Graph Theory.

Several natural questions arise from this work. Most importantly, one may ask about ways of adapting this approach to locating eigenvalues of general graph matrices, i.e., of matrices $M=(m_{ij})$ of order $n$ such that there is a graph $G$ on the vertex set $\{1,2,\ldots,n\}$ with the property that $m_{ij} \neq 0$ if and only if $\{i,j\} \in E(G)$. It is easy to see that essentially the same algorithm would work for $M$ (and still run in time $O(k^3 n)$) if the following conditions hold:
\begin{itemize}

\item[(i)] All nonzero off-diagonal entries of $M$ have the same value.

\item[(ii)] The algorithm has (oracle) access to the diagonal entries of $M$, that is, it is allowed to query directly the diagonal entry of a node produced by the slick $k$-expression.

\end{itemize}
Indeed, condition (i) ensures that the crucial relationship between the small submatrix that is processed at each stage and the actual matrix being diagonalized is preserved (see equations~\eqref{eq:matrixBQ},~\eqref{eq:formofM} and ~\eqref{eq:bigmatrixBQ}), while condition (ii) allows us to initialize the matrices at each leaf node of the parse tree. For instance, this would allow us to locate the eigenvalues of the Laplacian and of the signless Laplacian matrices associated with a graph as long as, together with the slick $k$-expression, the algorithm were given the degree of each vertex $v$ when $i(v)$ appears in the expression. (Of course, it is possible to compute, in polynomial-time, the degree of each vertex of a graph given by a $k$-expression, but we do not know how to do this in time $O(\poly(k)n)$, though time $O(n+m)$ is possible.) A nice open question would be to adapt this approach to matrices whose off-diagonal values may be assigned more values. More generally, it would certainly be interesting to find extensions of this work to other width parameters, as well as to the computation of other related matrix tasks.

Another natural research direction would be to further investigate the notion of slick clique-width. Even though it was introduced here with the sole objective of simplifying the description of the algorithm, the slick-clique width may be interesting for its own sake. As mentioned in the introduction, the graphs with slick clique-width equal to 1 are precisely the cographs. In general, If $H$ is an induced subgraph of $G$, then $\scwop(H) \leq \scwop(G)$ (for $H$, use the $k$-expression of $G$, ignoring the generation of any vertices in $V(G)-V(H)$. Remove all unnecessary operations from the expression.) Of course, this does not happen for general subgraphs, as all complete graphs are cographs and therefore satisfy $\scwop(K_n)=1$. So, given $k$, it makes sense to ask for the set $\mathcal{F}_k$ of all graphs $G$ such that $\scwop(G)>k$, but $\scwop(H)\leq k$ for any proper induced subgraph $H$ of $G$. In other words, the elements of $\mathcal{F}_k$ are the minimal graphs that cannot be represented by a slick $k$-expression (minimality is with respect to the partial order given by induced subgraphs). It is common to characterize graph classes through their set of minimal forbidden substructures, and the theory of cographs tells us that $\mathcal{F}_1=\{P_4\}$. The situation for $k=2$ is already much more complicated: a simple argument shows that $\scwop(T) \leq  2$ for any tree $T$, and it is also possible to show that distance hereditary graphs have slick clique-width at most 2. However, for cycles $C_n$, we have
\begin{eqnarray*}
\scwop(C_n)=
\begin{cases}
1, & \textrm{ if } n \in \{3,4\};\\
2, & \textrm{ if } n \in \{5,6\};\\
3 & \textrm{ if } n \geq 7.
\end{cases}
\end{eqnarray*}
This already implies that the set $\mathcal{F}_2$ is infinite, as it contains $C_n$ for any $n \geq 7$.

Finally, we think that it is possible to derive theoretical results about the location of eigenvalues for special classes of graphs through the analysis of the behavior of
our algorithm on these graphs. Of course, this would require a better description of slick expressions that define graphs in such a class.

\bibliographystyle{abbrv}
\bibliography{EV-clique-width}

\end{document}